\documentclass[10pt,reqno,helvetica]{amsart}
\usepackage{amsmath, amsthm, amsfonts}
\usepackage{hyperref}
\hypersetup{hidelinks,backref=true,pagebackref=true,hyperindex=true,colorlinks=true,citecolor=red,linkcolor=blue,breaklinks=true,urlcolor=ocre,bookmarks=true,bookmarksopen=false,pdftitle={Title},pdfauthor={Author}}

 \usepackage[mediumspace,mediumqspace,Grey,squaren]{SIunits}

\usepackage[usenames, dvipsnames]{xcolor}

\usepackage{pgf,tikz}
\usepackage{amsmath,amsthm,amssymb}
\usepackage{amscd,indentfirst,epsfig}
\usepackage{latexsym}
\usepackage{times}
\usepackage{enumerate}
\usepackage{mathrsfs}
\usepackage{stmaryrd}
\usepackage{amsopn}
\usepackage{amsmath}
\usepackage{amssymb}
\usepackage{amsfonts,bm,dsfont}

\usepackage{dsfont,mathtools}
\usepackage{enumerate,enumitem}
\usepackage{mathrsfs,esint}
\usepackage{stmaryrd}

\usepackage{amsopn,bm}
\usepackage{amsmath}
\usepackage{amssymb}
\usepackage{amsfonts}
\usepackage{amsbsy}
\usepackage{amscd,indentfirst,epsfig}
\usepackage{amsfonts,amsmath,latexsym,amssymb,verbatim,amsbsy,textcomp}
\usepackage{amsthm}
\usepackage{dsfont}
\usepackage{colordvi}
\usepackage{pstricks}
\usepackage{csquotes}%%%%%%%%%%%%%%%%%%%%%%%%%%%%%%%%%%%%%%%%%%%%%%%%%%%%%%%%%%%%%%%%%%%%%%%%%%%%%%%%%%%%%%%%%%%%%%%%%%%%%%%%%%%%%%%%%%%%%%%%%%%%%%%%%%%%%%
\def \leq {\leqslant}                                                                %%
\def \geq {\geqslant}                                                                %%
\def\ind#1{\lower5pt\hbox{$\scriptstyle #1$}}                                        %%
\def \d {\mathrm{d}}                                                                 %%
\def \R{\mathbb R}                                                                   %%
\def\S{\mathbb S}                                                                    %%
\def\Z{\mathbb Z}                                                                    %% 
\def\N{\mathbb N}                                                                    %%
\def \C{\mathbb C}																	 %%					
\def\A{\mathcal A}																	 %%	
\def\D{\mathscr D}																	 %%
\def\M{\mathcal M}																	 %%
\def \T {\mathbb{T}}                                                                 %%
\numberwithin{equation}{section}                                                     %%
\newcommand{\vertiii}[1]{{\left\vert\kern-0.25ex\left\vert\kern-0.25ex\left\vert #1  %%
    \right\vert\kern-0.25ex\right\vert\kern-0.25ex\right\vert}}                      %%
\newcommand{\verti}[1]{{\left\vert\kern-0.25ex\left\vert\kern-0.25ex\left\vert #1    %%
    \right\vert\kern-0.25ex\right\vert\kern-0.25ex\right\vert}}						 %%	
\def\e{\varepsilon}      															    %%	
\def \m {\bm{\varpi}}															        %%
\def \ho {h^{0}} 																	    %%
\def \hu {h^{1}}																		%%
\def \Rs {\mathcal{R}}																	%%
\def \ra {\Big\rangle}																	%%
\def \la {\Big\langle}																	%%
\def \en {\vartheta}																	%%
\def \vE {\theta}																		%%	
\def \re {\alpha}																		%%
\def\E{\mathcal{E}} 																	%%
\def \B{\mathcal{B}}                                         							%% 
\def \Y {\mathbb{Y}}                                         							%%
\def\W {\mathbb{W}} %\mathbb{W} 														%%	
\def \H {\mathcal{H}}                                        							%%
\def \vb {v_\ast}                                           							%%
\def \Q{\mathcal{Q}}                                        							%%
\def \LL {\mathscr{L}_{\re}} 							 							    %%	
\def \G {\mathcal{G}}							 										%%	
\newcommand\blfootnote[1]{%
  \begingroup
  \renewcommand\thefootnote{}\footnote{#1}%
  \addtocounter{footnote}{-1}%
  \endgroup
}

\newtheorem{theo}{Theorem}[section]                          							%%
\newtheorem{prop}[theo]{Proposition}                         							%%
\newtheorem{cor}[theo]{Corollary}                            							%%
\newtheorem{lem}[theo]{Lemma}                                							%%
\newtheorem{hyp}[theo]{Assumption}                    									%% 
\newtheorem{nb}[theo]{Remark}                                							%%
\begin{document}
%%%%%%%%%%%%%%%%%%%%%%%%%%%%%%%%%%%%%%%%%%%%%%%%%%%%%%%%%%%%%%%%%%%%%%%%%%%%%%%%%%%%%%%%%%%%%%%%%%%%%%%%%%%%%%%%%%%%%%%%%%
%%%%%																													%%		
%%%%%													TITLE															%%
%%%%%%%%%%%%%%%%%%%%%%%%%%%%%%%%%%%%%%%%%%%%%%%%%%%%%%%%%%%%%%%%%%%%%%%%%%%%%%%%%%%%%%%%%%%%%%%%%%%%%%%%%%%%%%%%%%%%%%%%%%

\title[From Boltzmann equation for granular gases to a modified Navier-Stokes-Fourier system]{From Boltzmann equation for granular gases to a modified \\ Navier-Stokes-Fourier system}

\author{Ricardo J. {\sc Alonso}}
\address{$^1$Texas A\&M University at Qatar, Science Department, Education City, Doha, Qatar.}
\address{$^2$Departamento de Matem\'atica, PUC-Rio, Rio de Janeiro, Brasil.} 
 \email{ricardo.alonso@qatar.tamu.edu}

\author{Bertrand {\sc Lods}}

\address{Universit\`{a} degli Studi di Torino \& Collegio Carlo Alberto, Department of Economics and Statistics, Corso Unione Sovietica, 218/bis, 10134 Torino, Italy.}\email{bertrand.lods@unito.it}

\author{Isabelle {\sc Tristani}}
\address{D\'epartement de Math\'ematiques et Applications, \'Ecole Normale Sup\'erieure, CNRS, PSL University, 75005 Paris, France.}\email{isabelle.tristani@ens.fr}

%%%%%%%%%%%%%%%%%%%%%%%%%%%%%%%%%%%%%%%%%%%%%%%%
\maketitle
\vspace{-0.8cm}
\begin{abstract}
In this paper, we give an overview of the results established in~\cite{ALT} which provides the first rigorous derivation of hydrodynamic equations from the Boltzmann equation for inelastic hard spheres in 3D. In particular, we obtain a new system of hydrodynamic equations describing granular flows and prove existence of classical solutions to  the aforementioned system.
One of the main issue is to identify the correct relation between the restitution coefficient (which quantifies the rate of energy loss at the microscopic level) and the Knudsen number which allows us to obtain non trivial hydrodynamic behavior. In such a regime, we construct strong solutions to the inelastic Boltzmann equation, near thermal equilibrium whose role is played by the so-called homogeneous cooling state. We prove then the uniform exponential stability with respect to the Knudsen number of such solutions, using a spectral analysis of the linearized problem combined with technical a priori nonlinear estimates. Finally, we prove that such solutions converge, in a specific weak sense, towards some hydrodynamic limit that depends on time and space variables only through macroscopic quantities that satisfy a suitable modification of the incompressible Navier-Stokes-Fourier system.

\vspace{0.3cm}

\end{abstract}

\blfootnote{\noindent {\bf Acknowledgements.}    RA gratefully acknowledges the support from O Conselho Nacional de Desenvolvimento Cient\'ifico e Tecnol\'ogico, Bolsa de Produtividade em Pesquisa  - CNPq (303325/2019-4).  BL gratefully acknowledges the financial support from the Italian Ministry of Education, University and Research (MIUR), ``Dipartimenti di Eccellenza'' grant 2018-2022. IT thanks the ANR EFI:  ANR-17-CE40-0030  and the ANR SALVE: ANR-19-CE40-0004 for their support. 
}


\hypersetup{linkcolor=red}
%\tableofcontents

%%%%%%%%%%%%%%%%%%%%%%%%%%%%%%%%%%%%%%%%%%%%%%%%

\section{Introduction}

In this paper, we report on some recent results obtained in~\cite{ALT} about the problem of deriving rigorously some hydrodynamic limit from the Boltzmann equation for inelastic hard spheres with small inelasticity. Our aim here is to give an account of the main aspects of our work~\cite{ALT} in a shorter -- \emph{reader-friendly} -- version that includes the main results as well as the main ideas and arguments. We shall only sketch the proofs of our results, referring the reader to \cite{ALT} for complete versions and details. 

%-------------%-------------%-------------%-------------%-------------%-------------%-------------%-------------%
\subsection{The problem}
\subsubsection*{\textit{\textbf{The kinetic model}}} We consider here the (freely cooling) Boltzmann equation which provides a statistical description of  identical smooth hard spheres suffering binary and \emph{inelastic collisions}:
	\begin{equation}\label{Bol-}
	\partial_{t}F + v\cdot \nabla_{x} F=\Q_{\re}(F,F)
	\end{equation}
supplemented with initial condition $F(0,x,v)=F^{\mathrm{in}}(x,v)$, where $F=F(t,x,v)$ is the density of granular gases having position $x \in \T_\ell^{d}$ and velocity $v \in \R^{d}$ at time $t\geq0.$ We consider here for simplicity the case of {flat torus} 
	\begin{equation}\label{torus}
	\T_{\ell}^{d}=\R^{d}\slash (2\pi\,\ell\,\Z)^{d}
	\end{equation}
for some typical length-scale $\ell >0$. 
%This corresponds to periodic boundary conditions:
%$$F(t,x+2\pi\,\ell\bm{e}_{i},v)=F(t,x,v), \qquad i=1,\ldots,d$$
%where $\bm{e}_{i}$ is the $i$-th vector of the canonical basis of $\R^{d}.$ 
The so-called restitution coefficient $\alpha$ belongs to $(0,1]$ and the collision operator $\Q_{\re}$ is defined in weak form as
	\begin{equation}\label{co:weak}
		\begin{split}
		 \int_{\R^{d}} \Q_{\re} (g,f)(v)\, \psi(v)\, \d v  =  \frac{1}{2}  \int_{\R^{2d}} f(v)\,g(v_{\ast})\,|v-v_{\ast}|
		\mathcal{A}_{\re}[\psi](v,v_{\ast})\, \d v_{\ast}\, \d v,
		\end{split}
	\end{equation}
where
	\begin{equation}    \label{coll:psi} 
	\mathcal{A}_{\re}[\psi](v,v_{\ast}) :=
	    \int_{\S^{d-1}}(\psi(v')+\psi(v_{\ast}')-\psi(v)-\psi(v_{\ast}))\, b(\sigma \cdot \bar{q})\, \d{\sigma},
	\end{equation}
and the post-collisional velocities $(v',v_{\ast}')$ are given
by
	\begin{equation}
		\begin{split} \label{co:transf}
		  v'=v+\frac{1+\re}{4}\,(|q|\sigma-q),&
		\qquad  v_{\ast}'=v_{\ast}-\frac{1+\re}{4}\,(|q|\sigma-q),\\
		\text{where} \qquad q=v-v_{\ast},& \qquad \bar{q}=q/|q|.
		\end{split}
	\end{equation}
Here, $\d\sigma$ denotes the Lebesgue measure on $\S^{d-1}$ and the angular part $b=b(\sigma \cdot \bar{q})$ 
of the collision kernel appearing in \eqref{coll:psi} is a non-measurable mapping integrable over $\S^{d-1}$.  There is no loss of generality assuming
	$$
	\int_{\S^{d-1}}b(\sigma \cdot \bar{q})\, \d{\sigma}=1, \qquad \forall \, \bar{q} \in \S^{d-1}.
	$$
Notice that one can also give a strong formulation of the collision operator~$\Q_\alpha$ (see~\cite[Appendix~A]{ALT}). This strong formulation is simpler in the elastic case ($\alpha=1$), we here give it for later use:
	\begin{equation} \label{eq:Q1strong}
	\Q_1(g,f)(v) 
	= \int_{\R^d \times \S^{d-1}} \left( g(v_\ast') f(v') - g(v_\ast) f(v)\right) |v-v_\ast| \,b(\sigma \cdot \bar{q})\, \d{\sigma}\, \d v_\ast.
	\end{equation}
The true definition actually involves pre-collisional velocities and not post-collisional velocities $v'$ and $v'_\ast$ but they match in the elastic case, which explains the formula~\eqref{eq:Q1strong}. 

The fundamental distinction between the classical elastic Boltzmann equation and the associated to granular gases lies in the role of the parameter $\re \in (0,1)$, the coefficient of restitution that we suppose constant.  This coefficient is given by the ratio between the magnitude of the normal component (along the line of separation between the centers of the two spheres at contact) of the relative velocity after and before the collision. %(see Appendix \ref{Sec21} for the detailed microscopic velocities). 
The case $\re = 1$ corresponds to perfectly elastic collisions  where kinetic energy  is conserved. However, when $\re  < 1$, part of the kinetic energy of the relative motion is lost since
	$$
	|v'|^{2}+|\vb'|^{2}-|v|^{2}-|\vb|^{2}=-\frac{1-\re^{2}}{4}|q|^{2}\,\left(1-\sigma \cdot \bar{q}\right) \leq 0.
	$$
%It is assumed in this work that $\alpha$ is independent of the relative velocity $q$ (refer to \cite{A},~\cite{ALCMP}, and \cite{ALT} for the viscoelastic restitution coefficient case).  
Notice that the microscopic description \eqref{co:transf} preserves the momentum
	$$
	v'+\vb'=v+\vb
	$$
and, taking $\psi=1$ and then $\psi=v$ in \eqref{co:weak} yields the following conservation of macroscopic density and bulk velocity defined 
as	$$
	\bm{R}(t):=\int_{\T^{d}_{\ell} \times \R^{d}}F(t,x,v)\, \d v\, \d x
	\quad \text{and} \quad 
	\bm{U}(t) := \int_{\T^{d}_{\ell} \times \R^{d}}v F(t,x,v)\, \d v\, \d x,
	$$
for some solution $F(t,x,v)$ to~\eqref{Bol-}:
	$$
	\dfrac{\d}{\d t}\bm{R}(t)= \frac{\d}{\d t}\bm{U}(t)=0.
	$$
Consequently, there is no loss of generality in assuming that
	$$
	\bm{R}(t)=\bm{R}(0)=1, \qquad \bm{U}(t)=\bm{U}(0)=0, \qquad \forall \, t \geq0.
	$$
The main contrast between elastic and inelastic gases is that in the latter the \emph{granular temperature},
	$$
	\bm{T}(t):=\frac{1}{|\T^{d}_{\ell}|}\int_{\R^{d}\times \T^{d}_{\ell}}|v|^{2}F(t,x,v)\, \d v\, \d x
	$$
is constantly decreasing
	$$
	\dfrac{\d}{\d t}\bm{T}(t)=-(1-\re^{2})\mathcal{D}_{\re}(F(t),F(t)) \leq 0,
	$$
where $\mathcal{D}_{\re}(\cdot,\cdot)$ denotes the normalised energy dissipation associated to $\Q_{\re}$, see \cite{MiMo2}, given by
	\begin{equation}\label{eq:Dre}
	\mathcal{D}_{\re}(g,g):=
	\frac{\gamma_{b}}{4}\int_{\T^{d}_{\ell}}\frac{\d x}{|\T^{d}_{\ell}|}\int_{\R^d \times \R^d}g(x,v)g(x,\vb)|v-\vb|^{3}\, \d v\, \d \vb, 
	\end{equation}
where $\gamma_b$ is a positive constant depending only on the angular kernel $b$. 
%In fact, it is possible to show that 
%$$\lim_{t\to\infty}\bm{T}(t)=0$$
%which expresses the \emph{total cooling of granular gases}.  Determining the exact dissipation rate of the granular temperature is an important question known as \emph{Haff's law} \cite{haff}. 

%-------------%-------------%-------------%-------------%-------------%-------------%-------------%-------------%
\subsubsection*{\textit{\textbf{The problem of hydrodynamic limits}}}  To capture some hydrodynamic behaviour of the gas,  we need to write the above equation in \emph{nondimensional form} introducing the dimensionless Knudsen number which is proportional to the mean free path between collisions. We then introduce the classical Navier-Stokes rescaling of time and space (see~\cite{BaGoLe1}) to capture the hydrodynamic limit and introduce the particle density
	\begin{equation}\label{eq:Scaling}
	F_{\e}(t,x,v):=F\left(\frac{t}{\e^{2}},\frac{x}{\e},v\right), \qquad t \geq 0.
	\end{equation}
In this case, we choose for simplicity $\ell=\e$ in \eqref{torus} which ensures now that $F_{\e}$ is defined on $\R^{+}\times \T^{d}\times \R^{d}$ with $\T^{d}:=\T_{1}^{d}$. Under such a scaling, $F_{\e}$ satisfies the rescaled Boltzmann equation
	\begin{subequations}
	\begin{equation}\label{Bol-e}
	\e^{2}\partial_{t}F_{\e} + \e\,v\cdot \nabla_{x} F_{\e}
	=\Q_{\re}(F_{\e},F_{\e}) \quad \text{on} \quad \T^{d}\times\R^{d},
	\end{equation}
supplemented with the initial condition
	\begin{equation}\label{eq:init}
	F_{\e}(0,x,v)=F^{\mathrm{in}}_{\e}(x,v):=F^{\mathrm{in}}(\tfrac x\e, v).
	\end{equation}
	\end{subequations}
Conservation of mass and density is preserved under this scaling, if $F_\e$ solves~\eqref{Bol-e}, then
	$$
	\frac{\d}{\d t} \bm{R}_{\e}(t) = \frac{\d}{\d t} \bm{U}_{\e}(t) =0
	$$
where
	$
	\bm{R}_{\e}(t):=\int_{\T^d \times \R^{d}}F_{\e}(t,x,v)\, \d v\, \d x$ and 
	$
	\bm{U}_{\e}(t):=\int_{\T^d \times\R^{d}}F_{\e}(t,x,v)v\, \d v\, \d x, 
	$
whereas the cooling of the granular gas is given by the equation
	\begin{equation}\label{eq:TeHaff}
	\frac{\d}{\d t}\bm{T}_{\e}(t)= -\frac{1-\re^{2}}{\e^{2}}\mathcal{D}_{\re}(F_{\e}(t),F_{\e}(t)),
	\end{equation}
where  $\bm{T}_{\e}(t):=\frac{1}{|\T^{d}|}\int_{\T^d \times\R^{d}}|v|^{2}F_{\e}(t,x,v)\, \d v\, \d x$ and we recall that $\mathcal{D}_\re$ is defined in~\eqref{eq:Dre}.
The conservation properties of the equation imply that there is no loss of generality assuming that
	$$
	\bm{R}_{\e}(t) = 1, \quad \bm{U}_{\e}(t) = 0, \quad \forall \,\e >0,\, t \geq 0. 
	$$

In order to understand the free-cooling inelastic Boltzmann equation~\eqref{Bol-e}-\eqref{eq:init}, we perform a {\em self-similar change of variables}, which allows us to introduce an intermediate asymptotic and ensures that our equation has a non trivial steady state (see~\cite{MiMo1,MiMo2,MiMo3} for more details). After this change of variables, we are led to study the equation
	\begin{equation}\label{BE0}
	\e^{2}\partial_{t} f_{\e}+\e v \cdot \nabla_{x} f_{\e}+ (1-\re)\,\nabla_{v}\cdot (v f_{\e})
	=  \Q_{\re}(f_{\e},f_{\e}),
	\end{equation}
with initial condition 
	$$
	f_{\e}(0,x,v)=F_{\e}^{\mathrm{in}}(x,v).
	$$ 
Note that the drift term acts as an energy supply which prevents the total cooling down of the gas. It has been shown that there exists a \emph{spatially homogeneous} steady state $G_{\re}$ to~\eqref{BE0}. More specifically, there exists $\re_{0} \in (0,1)$ (where $\re_0$ is an explicit threshold value) such that for $\re \in (\re_{0},1)$, there exists a unique distribution $G_{\re}=G_\re(v)$ satisfying
	\begin{equation} \label{def:Gre}
	(1-\re) \nabla_{v}\cdot (v \, G_{\re})=  \Q_{\re}(G_{\re},G_{\re})
	\quad \text{with} \quad 
	\int_{\R^{d}}G_{\re}(v) \begin{pmatrix} 1 \\ v \end{pmatrix} \, \d v=\begin{pmatrix} 1 \\ 0 \end{pmatrix}.
	\end{equation}
Moreover, there exists some constant $C>0$ independent of $\alpha$ such that
	\begin{equation}\label{eq:GretoM}
	\|G_{\re}-\M\|_{L^{1}_v(\langle v \rangle^2)} \leq C (1-\alpha) %\underset{\alpha \to 1}{\longrightarrow} 0,
	\end{equation}
where $\M$ is the Maxwellian distribution
	\begin{equation}\label{eq:max}
	\M(v) :=(2\pi\en_{1})^{-d/2}\exp\left(-\frac{|v|^{2}}{2\en_{1}}\right), \qquad v \in \R^{d},
	\end{equation}
for some explicit temperature $\en_{1} >0$.  The Maxwellian distribution $\M$ is a steady solution for $\re=1$ and its prescribed temperature $\en_{1}$ (which ensures \eqref{eq:GretoM} to hold) will play a role in the rest of the analysis. %We refer to Appendix \ref{appen:homog} for more details and explanation of the role of $\en_{1}$.

 It is important to emphasize that, in all the sequel, all the threshold values on $\e$ and the various constants involved are actually depending \emph{only} on this initial choice. 

In order to reach some incompressible Navier-Stokes type equation in the limit $\e \to 0$, we introduce the following fluctuation $h_{\e}$ around the equilibrium $G_{\re}$:
	$$
	f_{\varepsilon}(t,x,v)=G_{\re}(v)+\varepsilon \,h_{\varepsilon}(t,x,v).
	$$
Our problem boils down to look at the following equation on $h_{\e}$:
	\begin{equation}\label{eq:BE}
	\begin{cases}
	&\partial_{t} h_{\varepsilon}+\dfrac{1}{\e} v \cdot \nabla_{x} h_{\varepsilon}
	=\dfrac{1}{\e^{2}} \LL h_{\varepsilon}
	+\dfrac{1}{\e} \Q_{\re}(h_{\varepsilon},h_{\varepsilon})\\[7pt]
	& h_{\e}(t=0)=h_\e^{\mathrm{\mathrm{in}}}:= \dfrac{1}{\e} (F_{\e}^{\mathrm{in}} - G_\re),
	\end{cases}
	\end{equation}
where $\LL$ is the linearized collision operator (local in the $x$-variable) defined as
	\begin{equation} \label{def:LL}
	\LL h:=\Q_{\re}(G_\re,h)+\Q_{\re}(h,G_\re) - (1-\re)\nabla_{v}\cdot (vh).
	\end{equation}
We also denote by $\mathscr{L}_{1}$ the linearized operator around $G_{1}=\M$, that is,
	\begin{equation} \label{def:LL1}
	\mathscr{L}_{1}h:=\Q_{1}(\M,h)+\Q_{1}(h,\M).
	\end{equation}
From now on, we will always assume that 
	\begin{equation} \label{eq:Fin}
	{\int_{\T^d \times \R^{d}}F^{\mathrm{in}}_{\e}(x,v)\left(\begin{array}{c}1 \\v \\|v|^{2}\end{array}\right)\, \d v \, \d x}
	=\left(\begin{array}{c}1 \\ 0 \\  {E_\e}\end{array}\right) \quad {\text{with} \quad E_\e>0 \quad \text{and} \quad \frac{E_\e - d\vartheta_1}{\e} \xrightarrow[\e \to 0]{}0}. 
	\end{equation}
{The choice of prescribing as initial energy some constant $E_\e>0$ satisfying $\e^{-1} (E_\e-d\vartheta_1) \to 0$ as $\e \to 0$ for our problem is natural because $d\vartheta_1$ is the energy of the Maxwellian $\M$ introduced in~\eqref{eq:max} and as we shall see later on, the restitution coefficient~$\re$ is intended to tend to $1$ as $\e$ goes to $0$ in our analysis (see~\eqref{eq:initialenergy}).} 
It is also worth noticing that assumption~\eqref{eq:Fin} and~\eqref{def:Gre} result in 
	$$
	\int_{\T^d \times \R^{d}}h_\e^{\mathrm{in}}(x,v)  \begin{pmatrix} 1 \\ v \end{pmatrix} \, \d v \, \d x=\begin{pmatrix} 0 \\ 0 \end{pmatrix}.
	$$
Moreover, equation~\eqref{eq:BE} preserves mass and {\em vanishing} momentum since, if $h_\e$ solves~\eqref{eq:BE}, then one formally has 
	\begin{equation} \label{eq:driftmomentum}\begin{split}
	\frac{\d}{\d t} \int_{\T^d \times \R^{d}} h_\e(t,x,v) v \, \d v \, \d x &=\int_{\T^d \times \R^{d}} \nabla_v \cdot (v h_{\e}(t,x,v)) v \,  \d v \, \d x 
	=  - \int_{\T^d \times \R^{d}} h_\e(t,x,v) v \, \d v \, \d x.
	\end{split}	\end{equation}
Consequently, there is no loss of generality assuming that 
	\begin{equation} \label{eq:conservh}
	\int_{\T^d \times \R^{d}}h_\e(t,x,v)  \begin{pmatrix} 1 \\ v \end{pmatrix} \, \d v \, \d x
	=\begin{pmatrix} 0 \\ 0 \end{pmatrix},
	\quad \forall \, t \geq 0. 
	\end{equation}

%-------------%-------------%-------------%-------------%-------------%-------------%-------------%-------------%
\subsubsection*{\textbf{\textit{Relation between the restitution coefficient and the Knudsen number}}}
The central underlying assumption in our study is the following relation between the restitution coefficient and the Knudsen number.
\begin{hyp}\label{hyp:re}
The restitution coefficient $\re(\cdot)$ is a continuously decreasing function of the Knudsen number $\e$ satisfying the scaling behaviour
	\begin{equation}\label{eq:scaling}
	\re(\e)=1-\e^{2} (\lambda_0 + \eta(\e))
	\end{equation}
with $\lambda_{0} \geq 0$ and some function $\eta(\cdot)$ that tends to $0$ as $\e$ goes to $0$. If $\lambda_0=0$, we assume furthermore that there exists $\e_\star>0$ such that $\eta(\cdot)$ is positive on $(0, \e_\star)$. 
\end{hyp}
{Notice that under this assumption, the hypothesis made on the energy of the initial data in~\eqref{eq:Fin} implies that 
	\begin{equation} \label{eq:initialenergy}
	\int_{\T^d \times \R^d} h_\e^{\mathrm{in}}(x,v) \,|v|^2 \, \d v \, \d x \xrightarrow[\e \to 0]{} 0. 
	\end{equation}
Indeed, using~\eqref{eq:Fin} and Assumption~\ref{hyp:re} combined with~\eqref{eq:GretoM}, we obtain
	\begin{multline*}
	\int_{\T^d \times \R^d} h_\e^{\mathrm{in}}(x,v) |v|^2\, \d v \, \d x 
	= \frac1\e \int_{\T^d \times \R^d} \left(F_\e^{\mathrm{in}}(x,v)- G_{\alpha(\e)}(v)\right) |v|^2\, \d v \, \d x \\
	= \frac{E_\e - d \vartheta_1}{\e} + \frac1\e \int_{\T^d \times \R^d} \left(\M(v)- G_{\alpha(\e)}(v)\right) |v|^2\, \d v \, \d x
	\xrightarrow[\e \to 0]{} 0. 
	\end{multline*}}

Still under Assumption~\ref{hyp:re}, we formally obtain that if $\e \to 0$ in~\eqref{eq:BE}, then $h_\e \to \bm{h}$ with~$\bm{h} \in \operatorname{Ker} \mathscr{L}_1$ where~$\mathscr{L}_1$ is defined in~\eqref{def:LL1}. We recall that when seeing $\mathscr{L}_1$ as an operator acting only on velocity on the space~$L^2_v(\M^{-1/2})$, then
	$$
	\operatorname{Ker} \mathscr{L}_1 = \operatorname{Span} \{\M, v_1 \M, \cdots v_d \M, |v|^2 \M\}
	$$
and the projection $\bm{\pi}_0$ onto $\operatorname{Ker} \mathscr{L}_1$ is given by 
	\begin{equation}\label{def:pi0}
	\bm{\pi}_{0}(g) : =\sum_{i=1}^{d+2}\left(\int_{ \R^{d} }g\,\Psi_{i}\,\d v \right)\,\Psi_{i}\,\M,
	\end{equation}
where 
	\begin{equation} \label{def:Psii}
	\Psi_{1}(v):=1, \quad \Psi_{i}(v):=\frac{1}{\sqrt{\en_{1}}}v_{i-1}, \quad i=2,\ldots,d+1 
	\quad \text{and} \quad \Psi_{d+2}(v):=\frac{|v|^{2}-d\en_{1}}{\en_{1}\sqrt{2d}}.
	\end{equation}
We deduce formally that $\bm{h}$ takes the following form
	$$
	\bm{h}(t,x,v)=\left(\varrho(t,x)+u(t,x)\cdot v + \frac{1}{2}\vE(t,x)(|v|^{2}-d\en_{1})\right)\M(v)
	$$
with 
	\begin{multline} 
	\varrho(t,x) := \int_{\T^d \times \R^d} \bm{h}(t,x,v) \, \d v, 
	\quad
	u(t,x) := \frac{1}{\vartheta_1} \int_{\T^d \times \R^d} \bm{h}(t,x,v) v \, \d v, \\
	\vE(t,x) :=  \int_{\T^d \times \R^d} \bm{h}(t,x,v) \frac{|v|^2-d\vartheta_1}{\vartheta_1^2d}\, \d v. 
	\end{multline}

It is worth mentioning that a careful spectral analysis of the linearized collision operator~$\LL$ defined in~\eqref{def:LL} shows that unless one assumes $1-\re$ at least of order $\e^{2}$, the eigenfunction associated to the energy dissipation would explode and prevent some exponential stability for~\eqref{eq:BE} to hold true (see Theorem~\ref{theo:linear}). Actually, in our study, we will require $\lambda_0$ to be relatively small with respect to the spectral gap associated to the elastic linearized operator to ensure stability in the inelastic case. If one assumes $\lambda_0=0$ (for example, one could assume $1-\re$ of order~$\e^q$ with $q>2$), the effect of the inelasticity is too weak in the hydrodynamic scale and the expected model is the classical Navier-Stokes-Fourier system. In short, we are left with two cases:
\begin{enumerate}
\item[\underline{Case 1:}] If $\lambda_{0}=0$, the expected model is the classical Navier-Stokes-Fourier system.

\smallskip

\item[\underline{Case 2:}] If $0 < \lambda_{0}< \infty$ is small enough (compared to some explicit quantities), the cumulative effect of inelasticity is visible  in the hydrodynamic scale and we expect a different model to the Navier-Stokes-Fourier system accounting for that.  
\end{enumerate}

In this nearly elastic regime, the energy dissipation rate in the system happens in a controlled fashion since the inelasticity parameter is compensated accordingly to the number of collisions per time unit. 
%This process mimics viscoelasticity as particle collisions become more elastic as the collision dissipation mechanism increases in the limit $\e\to0$ (see Assumption \ref{hyp:re} below).  In this way, we are able to consider a re-scaling of the kinetic equation in which a peculiar intermediate asymptotic emerges and prevents the total cooling of the granular gas.
Other regimes can be considered depending on the rate at which kinetic energy is dissipated; for example, an interesting regime is the \emph{mono-kinetic} one which considers the extreme case of infinite energy dissipation rate.  In this way, the limit is formally described by enforcing a Dirac mass solution in the kinetic equation yielding the \emph{pressureless Euler system} (corresponding to sticky particles).  Such a regime has been rigorously addressed in the one-dimensional framework in the interesting contribution {\cite{jabin}}. It is an open question to extend such analysis to higher dimensions {since the approach of {\cite{jabin}} uses the so-called Bony functional which is a tool specifically tailored for~1D kinetic equations.}
%-------------%-------------%-------------%-------------%-------------%-------------%-------------%-------------%
\subsection{Notations and definitions}
Let us introduce some useful notations for functional spaces. For any nonnegative weight function $m\::\:\R^{d}\to \R^{+}$,  we define, for all $p >1$ the space~$L^{p}(m)$ through the norm
	$$
	\|f\|_{L^{p}(m)}:=\left(\int_{\R^{d}}|f(\xi)|^{p}m(\xi)^{p}\, \d\xi\right)^{1/p},
	$$
We also define, for $p \geq 1$
	$$
	\W^{k,p}(m)=\left\{f \in L^{p}(m)\;;\;\partial_{\xi}^{\beta}f \in L^{p}(m) \:\forall \, |\beta| \leq k\right\}
	$$
with the usual norm, i.e., for $k \in \N$:
	$$
	\|f\|_{\W^{k,p}(m)}^{p}=\sum_{|\beta| \leq k}\|\partial_{\xi}^{\beta}f\|_{L^{p}(m)}^{p}.
	$$
For $m \equiv 1$, we simply denote the associated spaces by $L^{p}$ and $\W^{k,p}$. Notice that all the weights we consider here will depend only on velocity, i.e. $m=m(v)$. We will also use the notation $\langle \xi \rangle := \sqrt{1+|\xi|^2}$ for $\xi \in \R^d$. 

On the complex plane, for any $a \in \R$, we set
	\begin{equation} \label{def:Ca}
	\C_{a}:=\{z \in \C\;;\;\mathrm{Re}\,z >-a\}, \qquad \C_{a}^{\star}:=\C_{a}\setminus\{0\}
	\end{equation}
and, for any $r >0$, we set
	$$
	\mathbb{D}(r)=\{z \in \C\;;\;|z| \leq r\}.
	$$

We also introduce the following notion of hypo-dissipativity in a general Banach space $(X,\|\cdot\|)$. A closed (unbounded) linear operator $A\::\:\D(A) \subset X \to X$ is said to be \emph{hypo-dissipative} on $X$ if there exists  a norm, denoted by~$\vertiii{\cdot}$, equivalent to the $\|\cdot\|$--norm such that $A$ is dissipative on the space $(X,\vertiii{\cdot})$, that is, 
	$$
	\vertiii{(\lambda-A)h} \geq \lambda\,\vertiii{h}, \qquad \forall \, \lambda >0,\,\;h \in \D(A).
	$$
Given two Banach spaces $X$ and $Y$, we denote with $\|\cdot\|_{X \to Y}$ the operator norm on the space of $\mathscr{B}(X,Y)$ linear and continuous operators from $X$ to $Y$.
%\begin{nb}  This is equivalent to the following (see Proposition 3.23, p. 88 in \cite{engel}):  if $\vertiii{\cdot}_{\star}$ denotes the norm on the dual space $X^{\star}$, for all  $h \in \D(A),$  there exists $\bm{u}_{h} \in X^{\star}$ such that
%	\begin{equation*}
%	\left[ \bm{u}_{h},h\right]=\vertiii{h}^{2}=\vertiii{\bm{u}_{h}}_{\star}^{2}\qquad \text{ and } \quad
%	\mathrm{Re}\left[\bm{u}_{h},A\,h\right] \leq 0,
%	\end{equation*}
%where $\left[\cdot\,,\cdot\right]$ denotes the duality bracket between $(X^{\star},\vertiii{\cdot}_{\star})$ and $(X,\vertiii{\cdot})$.
%\end{nb}

Note also that in what follows, for two positive quantities $A$ and $B$, we denote by~$A \lesssim B$ if there exists a universal positive constant $C$ (which is in particular independent of the parameters $\alpha$ and $\e$) such that~$A \leq CB$.
%-------------%-------------%-------------%-------------%-------------%-------------%-------------%-------------%
\subsection{Main results}
The main results are both about the solutions to \eqref{eq:BE}. The first one is the following Cauchy theorem regarding the existence and uniqueness of close-to-equilibrium solutions to \eqref{eq:BE}. 
The functional spaces at stake are ${L^{1}_{v}L^2_x}$-based Sobolev spaces $\E_{1} \hookrightarrow \E$ defined through 
	\begin{equation} \label{def:E-E1}
	\E:= {\W^{k,1}_{v}\W^{m,2}_{x}}(\langle v\rangle^{q}), \quad \E_{1}:={\W^{k,1}_{v}\W^{m,2}_{x}}(\langle v\rangle^{q+1})
	\quad \text{with} \quad
	 {m  > d}, \quad m-1 \geq  k \geq 0, \quad q   \geq 3. 
	\end{equation}

\begin{theo}\label{theo:main-cauchy}
Under Assumption \ref{hyp:re}, for~$\e,\lambda_0$ and $\eta_0$ sufficiently small (with explicit bounds), if~$h^{\rm in}_\e \in \E$ is such that	
	$$
	\|h^{\rm in}_\e\|_{\E} \leq \eta_0,
	$$
then the inelastic Boltzmann equation \eqref{eq:BE} has a unique solution 
	$$
	h_{\e} \in\mathcal{C}\big([0,\infty); \E\big) \cap L^1\big([0,\infty); \E_{1}\big)
	$$ 
satisfying for any $r \in (0,1)$,
	\begin{equation*}
	\left\|h_\e(t)\right\|_{\E}\leq C \,\eta_0\,\exp\left(-(1-r){\lambda}_{\e}\,t\right), 
	\qquad \forall \, t >0
	\end{equation*}
for some positive constant $C=C(r) >0$ independent of $\e$ and where ${\lambda}_{\e}\underset{\e\to 0}{\sim}\lambda_{0}+\eta(\e)$ with~$\lambda_0$ and~$\eta=\eta(\e)$ that have been introduced in Assumption~\ref{hyp:re}.
\end{theo}

\begin{nb}
It is worth pointing out that the close-to-equilibrium solutions we construct are shown to decay with an exponential rate as close as we want to~${\lambda}_{\e} \sim \frac{1-\re(\e)}{\e^{2}}$ (which is the energy eigenvalue of the linearized operator, see Theorem \ref{theo:linear} hereafter). The rate of convergence can thus be made uniform with respect to the Knudsen number~$\e$ (notice that if~$\lambda_0=0$, we obtain a rate of decay as close as we want to~$\eta(\e)$, we thus obtain a uniform bound in time but not a uniform rate of decay).  
\end{nb}

The estimates on the solution $h_{\e}$ provided by Theorem \ref{theo:main-cauchy} are enough to prove that the solution $h_{\e}(t)$ converges towards some hydrodynamic solution $\bm{h}$ which depends on $(t,x)$ only through macroscopic quantities $(\varrho(t,x),u(t,x),\vE(t,x))$ which are solutions to a suitable modification of the incompressible Navier-Stokes system. This is done under an additional assumption on the initial datum that is lightly restrictive. Before stating our main convergence result, we introduce the notation
	$$
	\mathscr{W}_\ell := \left(\W^{\ell,2}_x\left(\T^d\right)\right)^{d+2}, \quad \ell \in \N
	$$ 
and we furthermore assume that in the definition of the functional spaces~\eqref{def:E-E1}, the following conditions are satisfied:
	$$
	m  > d, \quad m-1 \geq  k \geq 1, \quad q   \geq 5.
	$$

\begin{theo}\label{theo:hydro}
We suppose that the assumptions of Theorem~\ref{theo:main-cauchy} are satisfied. We assume furthermore that there exists $(\varrho_{0},u_{0},\vE_{0}) \in \mathscr{W}_m$ such that
	$$
	\lim_{\e\to0}\left\|\bm{\pi}_{0}h_{\mathrm{in}}^{\e}-h_{0}\right\|_{ {L^{1}_{v}\W^{m,2}_{x}(\langle v \rangle^q)}}=0,
	$$
where we recall that $\bm{\pi}_0$ is the projection onto the kernel of $\mathscr{L}_1$ defined in~\eqref{def:pi0} and
	\begin{equation} \label{def:h0}
	h_{0}(x,v):=\left(\varrho_{0}(x)+u_{0}(x)\cdot v + \frac{1}{2}\vE_{0}(x)(|v|^{2}-d\en_{1})\right)\M(v).
	\end{equation}
Then, for any $T >0$, the family of solutions $\left\{ h_{\e} \right\}_{\e}$ constructed in Theorem~\ref{theo:main-cauchy} converges in some weak sense to a limit $\bm{h}=\bm{h}(t,x,v)$ which is such that 
	\begin{equation}\label{eq:hlimint}
	\bm{h}(t,x,v)=\left(\varrho(t,x)+u(t,x)\cdot v + \frac{1}{2}\vE(t,x)(|v|^{2}-d\en_{1})\right)\M(v),
	\end{equation}
where 
	$$
	(\varrho,u,\vE) \in \mathcal{C}\left([0,T]\,;\,\mathscr{W}_{m-1}\right) \cap L^2\left((0,T)\,;\,\mathscr{W}_m\right)
	$$
is solution to the following  \emph{incompressible Navier-Stokes-Fourier system with forcing}
	\begin{equation}\label{eq:NSFint}
	\begin{cases}
	\partial_{t}u-{\frac{{\nu}}{\en_1}}\,\Delta_{x}u + {\en_{1}}\,u\cdot \nabla_{x}\,u+\nabla_{x}p=\lambda_{0}u,\\[6pt]
	\partial_{t}\,\vE-\frac{\gamma}{\en_{1}^{2}}\,\Delta_{x}\vE + \en_{1}\,u\cdot \nabla_{x}\vE
	=\dfrac{\lambda_{0}\,\bar{c}}{2(d+2)}\sqrt{\en_{1}}\,\vE,\\[8pt]
	\mathrm{div}_{x}u=0, \qquad \varrho + \en_{1}\,\vE = 0,
	\end{cases}
	\end{equation}
subject to initial conditions $(\varrho_{\mathrm{in}},u_{\mathrm{in}},\vE_{\mathrm{in}})$ defined by 
	\begin{equation}\label{eq:DI}
	u_{\mathrm{in}}:=\mathcal{P}u_{0}, \quad 
	\vE_{\mathrm{in}}:=\frac{d}{d+2}\theta_{0}-\frac{2}{(d+2)\en_{1}}\varrho_{0}, \quad
	\varrho_{\mathrm{in}}:=-\en_{1}\vE_{\mathrm{in}}
	\end{equation}
where~$\mathcal{P}$ is the Leray projection on divergence-free vector fields and~$(\varrho_0,u_0,\theta_0)$ have been introduced in~\eqref{def:h0}. 
The viscosity~${\nu} >0$ and heat conductivity~$\gamma >0$ are explicit and~$\lambda_{0} >0$ is the parameter appearing in Assumption~\ref{hyp:re}. The parameter~$\bar{c} >0$ is depending on the collision kernel~$b(\cdot)$.
\end{theo}

\begin{nb} \label{nb:WP}
The data that we consider here are actually quite general. Indeed, the assumption that we make only tells that the macroscopic projection of $h^{\mathrm{in}}_\e$ converges towards some macroscopic distribution and we do not make any assumption on the macroscopic quantities of this distribution. Namely, we do not suppose that the divergence free and the Boussinesq relations are satisfied by $(\varrho_0,u_0,\theta_0)$, the initial layer that could be created by such a lack of assumption is actually absorbed in our notion of weak convergence, the precise notion of which being very peculiar and strongly related to the a priori estimates used for the proof of Theorem \ref{theo:main-cauchy} (see Theorem~\ref{theo:conv} for more details on the type of convergence). 
\end{nb}

To prove Theorem~\ref{theo:hydro}, our approach is reminiscent of the program established in \cite{BaGoLe1,BaGoLe2,golseSR,SR} but simpler because our solutions are stronger than the renormalized ones that are used in~\cite{golseSR}. It is based on computations and compactness arguments that were already used in the elastic case. Let us point out that in our case, additional terms appear due to the inelasticity and they can be handled in the framework of Assumption~\ref{hyp:re}. In Section~\ref{sec:hydro}, we present the proof but only mention its main steps and arguments (details can be found in~\cite[Section~6]{ALT}).

%%%%%%%%%%%%%%%%%%%%%%%%%%%%%%%%%%%%%%%%%%%%%%%%%%%%%%%%%%%
\section{Study of the kinetic linearized problem} \label{sec:linear}
\subsection{Main result on the linearized operator}
The first step in the proof of Theorem~\ref{theo:main-cauchy} is the spectral analysis of the linearized problem associated to~\eqref{eq:BE}. To that end, we introduce
	$$
	\G_{\re,\e}h:=-\frac{1}{\e} v \cdot \nabla_{x}h +\frac{1}{\e^2}\LL h.
	$$ 
We are going to state our main result on $\G_{\re,\e}$ in the space $\E$ defined in~\eqref{def:E-E1} but our analysis actually allows to treat even larger spaces (namely, we can obtain the same result under the softer constraints $m \geq k \geq 0$ and $q>2$) but we only state the linear result in this case because it is the only one that will be used in the rest of the paper. 
Let us also recall that, in any reasonable space (in particular in $\mathcal{E}$ and $\Y_j$ for~$j = -1,0,1$ defined in~\eqref{def:Y}-\eqref{def:Yj}), the elastic operator has a spectral gap: there exists~$\mu_{\star} >0$ such that
	\begin{equation} \label{eq:sgG1eps}
	\mathfrak{S}(\G_{1,\e})\cap \C_{\mu_\star}=\{0\}
	\end{equation}
where $0$ is an eigenvalue of algebraic multiplicity $d+2$ of $\G_{1,\e}$ associated to the eigenfunctions 
$$\{\M,v_{1}\M,\dots,v_d\M,|v|^{2}\M\}$$ (recall that $\C_{\mu_\star}$ is defined in~\eqref{def:Ca}). This can be proven by an enlargement argument due to~\cite{GMM} based on the fact that in the Hilbert space
	\begin{equation} \label{def:H}
	\H := \W^{m,2}_{x,v} (\M^{-1/2}), \quad m>d
	\end{equation}
a result of hypocoercivity has been proven in~\cite{briant} (the constraint $m \geq 1$ would actually be enough but we will only make use of this result for $m>d$ in the sequel). More precisely, introducing the other Hilbert space 
	\begin{equation} \label{def:H1}
	\H_1 := \W^{m,2}_{x,v} (\M^{-1/2} \langle v \rangle^{1/2}),
	\end{equation} 
there exists $\mu_\star>0$ and a norm equivalent to the usual one uniformly in $\e$ (we still denote it by $\|\cdot\|_\H$ and~$\langle \cdot, \cdot \rangle_\H$ its associated scalar product to lighten the notations) such that 
	\begin{equation} \label{eq:hypocoerc}
	\langle \G_{1,\e} h,h \rangle_\H \leq - \frac{\mu_\star}{\e^2} \|(\mathbf{Id} - \bm{\pi}_0) h\|^2_{\H_1} - \mu_\star \|h\|^2_{\H_1}.
	\end{equation}
%In the previous inequality, we have denoted by $h^\perp$ the microscopic part of~$h$, namely~$h^\perp := (\mathbf{Id} - \bm{\pi}_0) h$ where we recall that $\bm{\pi}_0$ is the projection onto the kernel of $\mathscr{L}_1$ that has been introduced in~\eqref{def:pi0}. 

As we shall see in the following result, the scaling \eqref{eq:scaling} in Assumption~\ref{hyp:re} is precisely the one which allows to preserve exactly $d+2$ eigenvalues in the neighborhood of zero for~$\G_{\re,\e}$.  Let us now state our main spectral result (see Figure~1 for an illustration where we have denoted $\lambda_\e:=-\lambda_{d+2}(\e)$):
%
%
%\begin{nb} Notice that the eigenvalue $\lambda_\e := -\lambda_{d+2}(\e)$ can be seen as the \emph{energy eigenvalue} in the sense that
%$$\int_{\R^{d}}\G_{\re,\e}\varphi(v)\,|v|^{2}\, \d v=-\lambda_{\e}\int_{\R^{d}}\varphi(v)|v|^{2}\, \d v$$
%for any smooth test-function $\varphi$.
%\end{nb}

\begin{figure}\label{fig:1}
\begin{tikzpicture}[line cap=round,line join=round,x=1.0cm,y=1.0cm, scale=0.9]
\draw[->,color=black] (-6,0) -- (4,0);
\draw[->,color=black] (0,-3.3) -- (0,3.3);
\draw[color=black] (+5.2pt,-5.3pt) node  {\tiny O} ;
\draw(0,0) circle (1cm);
\draw[fill](-5,0) circle (0.6mm);
\draw[fill](-4,0) circle (0.6mm);
\draw (1.6,-0.8) node {\small $\mathbb{D}(\mu_{\star}-\mu)$};
\draw[dashed] (-4,-3)--(-4,3);
\draw[dashed] (-5,-3)--(-5,3);
\draw (0,0)--(0.71,0.71)node[midway, above]{$\chi$};
\draw[<->] (-5,-3.3)--(-4,-3.3)node[midway, below]{$\chi$};
\draw (-5,-0.3) node[left] {$-\mu_{\star}$};
\draw (-4,-0.3) node[right] {$-\mu$};
\draw (4.3,0) node[below] {\small $\mathrm{Re}\lambda$};
\draw (0,3.3) node[left] {\small $\mathrm{Im}\lambda$};
\draw[fill] (-0.3,0) circle (0.6mm);
\draw (-0.1,-0.3) node[left] {\small $-{\lambda}_{\e}$};
\draw (-2,2.3) node {\small $\C_{\mu}\setminus\mathbb{D}(\mu_{\star}-\mu)$};
\end{tikzpicture}
\caption{The set $\C_{\mu} \setminus\mathbb{D}(\mu_{\star}-\mu)$ and the eigenvalue $-{\lambda}_{\e}$.}
\end{figure}

\begin{theo}\phantomsection \label{theo:linear} 
Assume that Assumption~\ref{hyp:re} is met. For $\mu$ close enough to $\mu_{\star}$ defined in~\eqref{eq:sgG1eps} (in an explicit way), there are some explicit~$\overline{\e}>0$ and $\overline \lambda>0$ depending only on $\chi:=\mu_\star-\mu$ such that, for all~$\e \in (0,\overline{\e})$ and~$\lambda_0 \in [0,\overline \lambda)$, the linearized operator $\G_{\re(\e),\e}$ has the following spectral property in~$\mathcal{E}$: 
\begin{equation}\label{eq:spectGe}
\mathfrak{S}(\G_{\re(\e),\e}) \cap\C_\mu=\{\lambda_{1}(\e),\ldots,\lambda_{d+2}(\e)\},
\end{equation}
with
$$\lambda_{1}(\e)=0, \qquad \lambda_{j}(\e)=\frac{1-\re(\e)}{\e^{2}}, \qquad j=2,\ldots,d+1,$$
and
$$\lambda_{d+2}(\e)=-\frac{1-\re(\e)}{\e^{2}}+\mathrm{O}(\e^{2}) \quad \text{as} \quad \e \to 0.$$
\end{theo} 

\begin{nb}
It is worth noticing that the eigenvalue $\lambda_1(\e)=0$ corresponds with the property of mass conservation of the operator $\G_{\re(\e),\e}$. Concerning the intermediate eigenvalues $\lambda_{j}(\e)$ for $j=2,\ldots,d+1$, as one can see on their definition, they may be positive, this is due to the fact that the collision operator $\Q_\re$ preserves momentum while the drift operator $\nabla_v ( v \cdot)$ does not. However, using~\eqref{eq:driftmomentum}, one can prove that the vanishing momentum is preserved by the whole operator $\G_{\re(\e),\e}$, consequently, those eigenvalues won't affect the long-time analysis of our problem. Finally, the eigenvalue $\lambda_{d+2}(\e)$ is directly linked to the non-preservation of energy property of~$\G_{\re(\e),\e}$. 
\end{nb}

We are going to prove Theorem~\ref{theo:linear} in two stages. First, we perform a perturbative argument (reminiscent of~\cite{Tr}) in a $L^2_{v,x}$-based Sobolev space, namely in
	\begin{equation} \label{def:Y}
	\Y:=\W^{s,2}_{v}\W^{\ell,2}_{x} (\langle v \rangle^{r}), \quad \ell \in \N, \, \, \, \, s \in \N^*, \, \, \, \, \ell \geq s+1, \, \, \, \, r>r^\star+ \kappa+2
	\end{equation}
where 
	$$
	r^\star := 4\sqrt{\frac{\sigma_{1}}{\sigma_{0}}} + \frac{3}{2}
	$$
with $\sigma_0$ and $\sigma_1$ defined in~\eqref{eq:collfreq} and $\kappa>d/2$. The key point of our approach is to see~$\G_{\re,\e}$ as a \emph{perturbation} of the elastic linearized operator $\G_{1,\e}$. We then use an enlargement argument (from~\cite{GMM}) to extend the result from~$\Y$ to the space $\mathcal{E}$ defined in~\eqref{def:E-E1}. 
%In what follows, we will use the following notations:
%\textcolor{red}{$$
%\X:=\W^{\ell,1}_{x}\W^{s,1}_{v} (\langle v \rangle^q), \qquad \Y:=\W^{\ell,1}_{x}\W^{s+1,1}_{v} (\langle v \rangle^{q+2})\,
%$$
%$$
%\qquad \ell,s\in \N, \quad {\ell \geq s+1}\quad  \text{ and }  \quad q >2
%$$}
%so that $\Y \hookrightarrow \X$. 

\smallskip
Several remarks are in order:
\begin{enumerate}
\item[(i)] First, let us remind that the global equilibrium of our equation~$G_\re$ defined in~\eqref{def:Gre} has some exponential fat tail and in particular, decays more slowly than a standard Maxwellian distribution (see~\cite{MiMo3}). As a consequence, we can not rely on classical works on the elastic linearized operator which are developed in spaces of type~$L^2_{v,x}(\M^{-1/2})$ with $\M$ defined in~\eqref{eq:max}. To overcome this difficulty, we exploit results coming from~\cite{GMM} in which an enlargement theory has been carried out. The results proven in~\cite{GMM} include a spectral analysis of the elastic Boltzmann operator $\G_{1,1}$ in larger spaces (in particular of type~$L^2_{v,x}$) with ``soft weights'' that can be polynomial or stretched exponential. In the same line of ideas, these results have been extended to the rescaled elastic operator $\G_{1,\e}$ in~\cite{bmam}.

\item[(ii)] Let us also point out that the perturbation at stake does not fall into the realm of the classical perturbation theory of unbounded operators as described in \cite{kato} because the perturbation is not relatively bounded. Indeed, the domain of $\G_{1,\e}$ in~$\Y$ is given by $\W^{s+1,2}_{v}\W^{\ell+1,2}_{x} (\langle v \rangle^{r+1})$ while if one wants to be sharp in terms of rate, the best estimate in terms of functional spaces that we are able to get is 
	\begin{equation} \label{eq:perturb}
	\|\G_{\re,\e} - \G_{1,\e}\|_{\Y_j \to \Y_{j-1}} ={1 \over \e^2} \|\mathscr{L}_{\re} - \mathscr{L}_{1}\|_{\Y_j \to \Y_{j-1}} \lesssim \frac{1-\alpha}{\e^2}, \quad j=0,1, 
	\end{equation}
where the spaces $\Y_j$ are defined through  
	\begin{equation} \label{def:Yj}
	\qquad \quad
	\Y_{-1} := \W^{s-1,2}_v\W^{\ell,2}_x (\langle v \rangle^{r-\kappa-2}), \quad 
	\Y_0 := \Y, \quad \Y_1 := \W^{s+1,2}_v \W^{\ell,2}_x ( \langle v \rangle^{r+\kappa+2})
	\end{equation}
with $\kappa>d/2$. 
These estimates (whose proofs can be found in~\cite[Lemma~3.3]{ALT}) are a generalization and optimisation of estimates obtained in~\cite{MiMo3} and are sharp in terms of rate, this sharpness being needed in our analysis since it allows us to deal with the case $\lambda_0>0$ in Assumption~\ref{hyp:re}. 

\item[(iii)] As a consequence, we have to use refined perturbation arguments whose key insights come from~\cite{Tr}. Note however that we drastically simplify the analysis performed in~\cite{Tr} by remarking that the difference operator $\G_{\re,\e} - \G_{1,\e}$ does not involve any spatial derivative and that we ``only'' need to develop a spectral analysis of $\G_{\re,\e}$ without being able to obtain decay properties on the associated semigroup. As a consequence, we don't need to use a spectral mapping theorem, nor do we need to use an iterated version of Duhamel formula and this is crucial in order to reach the optimal scaling~\eqref{eq:scaling} for our restitution coefficient. %Indeed, the reiteration of Duhamel formula is related to Dyson-Phillips iterates which lead to bad rates of decay in terms of $\e$. 

\item[(iv)] Let us finally mention that we perform our perturbative argument in $\Y$ which is a~$L^2_{v,x}$-based Sobolev space instead of performing it in $\mathcal{E}$ (which is $L^1_vL^2_x$-based) directly. This intermediate step seems necessary because even if $\LL - \mathscr{L}_1$ satisfies nice estimates in~$L^1_v$, the use of Fubini theorem is actually crucial to get the rate $(1-\alpha)/\e^2$ in estimates of type~\eqref{eq:perturb}. 
\end{enumerate}

\smallskip

\subsection{Elements of proof of Theorem~\ref{theo:linear}}
As mentioned above, the basis of the proof of this theorem is to see~$\LL$ as a perturbation of $\mathscr{L}_{1}$.

We start by giving a splitting of it into two parts: one which has some good regularizing properties (in the velocity variable) and another one which is dissipative. For any $\delta>0$, one can write 
$\mathscr{L}_1= \mathcal{A}^{(\delta)} + \mathcal{B}_1^{(\delta)}$ with $\mathcal{A}^{(\delta)}$ and $\mathcal{B}_1^{(\delta)}$ defined through an appropriate mollification-truncation process (see~\cite[Section~4.3.3]{GMM} and~\cite[Section~2.2]{ALT} for the details). The elastic collision operator $\mathscr{L}_1$ writes (see the strong formulation of~$\Q_1$ in~\eqref{eq:Q1strong}): 
	$$
	\mathscr{L}_1 g = \int_{\R^d \times \S^{d-1}} b(\sigma \cdot \bar{q}) |v-v_*| (\M'_* g' +\M' g'_* - \M g_*)\, \d{\sigma}\, \d v_*
	- \int_{\R^{d}}\M_*|v-v_{*}|\, \d v_{*} g
	$$
where we have used the shorthand notations $g = g(v)$, $g_* = g(v_*)$, $g' = g(v')$, $g'_* = g(v'_*)$. 
We define
	$$
	\mathcal{A}^{(\delta)}  g
	:= \int_{\R^d \times \S^{d-1}} \Theta_\delta \, b(\sigma \cdot \bar{q}) |v-v_*| (\M'_* g' +\M' g'_* - \M g_*)\, \d{\sigma}\, \d v_* 
	$$
	\begin{align*}
	\mathcal{B}_1^{(\delta)} g
	&:= \int_{\R^d \times \S^{d-1}} (1-\Theta_\delta) \, b(\sigma \cdot \bar{q}) |v-v_*| (\M'_* g' +\M' g'_* - \M g_*)\, \d{\sigma}\, \d v_* %\\
	%&\quad 
	- g\int_{\R^{d}}\M_*|v-v_{*}|\, \d v_{*}  
	\end{align*}
where $\Theta_\delta = \Theta_\delta(v,v_*,\sigma)$ is an appropriate truncature function.
%
%It is worth mentioning that the dissipativity property of $\mathcal{B}_1^{(\delta)}$ comes from the loss term of $\Q_1(\M,\cdot)$. If one writes $\Q_1(\M,\cdot) = \Q_1^+(\M,\cdot)- \Q_1^-(\M,\cdot)$ with obvious notations, one has 
%	\begin{equation}\label{eq:SigmaM}
%	\Q_1^-(\M,f) (v)= \int_{\R^{d}}\M(v_{*})|v-v_{*}|\d v_{*} f(v) =: \Sigma_M(v) f(v) , \qquad v \in \R^{d}.
%	\end{equation}
{The dissipativity property of $\mathcal{B}_1^{(\delta)}$ comes from the fact that the truncature function $\Theta_\delta$ is defined so that the first term is small as $\delta$ goes to $0$ and the fact that there exist $\sigma_0>0$ and $\sigma_1>0$ such that 
	\begin{equation} \label{eq:collfreq}
	\sigma_0  \langle v \rangle \leq \int_{\R^{d}}\M_{*}|v-v_{*}|\, \d v_{*} \leq \sigma_1  \langle v \rangle, \quad v \in \R^d. 
	\end{equation}}
As a consequence, $	\mathcal{B}_1^{(\delta)}$ is going to be dissipative for $\delta$ small enough. 
This leads to the following decomposition of $\LL$:
	\begin{equation} \label{eq:splitLalpha}
	\mathscr{L}_{\re}=\mathcal{B}_{\re}^{(\delta)} + \mathcal{A}^{(\delta)}\,,
	\qquad \text{where} \quad 
	\mathcal{B}_{\re}^{(\delta)}:=\underbrace{\mathcal{B}_{1}^{(\delta)}}_{\text{dissipative}}
	+\underbrace{\left[\LL-\mathscr{L}_{1}\right]}_{\text{small as $\alpha \to 1$}}
	\end{equation}
and then the following decomposition of $\G_{\re,\e}$:
	\begin{equation} \label{eq:splitGeps}
	\mathcal{G}_{\re,\e}=\mathcal{A}_{\e}^{(\delta)}+ \mathcal{B}_{\re,\e}^{(\delta)}\,, 
	\qquad \text{where} \qquad 
	\mathcal{A}_{\e}^{(\delta)}:=\frac{1}{\e^2} \A^{(\delta)}, \quad 
	\mathcal{B}_{\re,\e}^{(\delta)}:=\frac{1}{\e^2} \mathcal{B}_{\re}^{(\delta)} -\frac{1}{\e} v\cdot \nabla_{x}\,.
	\end{equation}

Our analysis of this splitting and then of the spectrum of $\mathcal{G}_{\re,\e}$ relies on several elements: 
the nice properties of the above-mentioned splitting of $\mathscr{L}_1= \mathcal{A}^{(\delta)} + \mathcal{B}_1^{(\delta)}$ coming from~\cite{GMM},
%\cite[Lemmas 4.12, 4.14 \&~4.16]{GMM}, 
some refined bilinear estimates on the collision operator coming from~\cite{ACG}, new estimates on $\G_\re-\M$ that are reminiscent of estimates proven in~\cite{CMS} (see~\cite[Lemma~2.3]{ALT}) and also new estimates on~$\Q_\re - \Q_1$ (see~\cite[Lemmas~2.1 and 2.2]{ALT}). Concerning the latter point, we exploit ideas developed in~\cite{MiMo3} but our situation is more involved because we work in polynomially weighted spaces whereas in~\cite{MiMo3}, the authors were working with stretched exponential weights.

In the following lemma, we provide some regularization and hypodissipativity results on the splitting $\mathcal{G}_{\re,\e}=\mathcal{A}_{\e}^{(\delta)}+ \mathcal{B}_{\re,\e}^{(\delta)}$ (see~\cite[Lemma~2.7 and Proposition~2.9]{ALT}):
 \begin{lem} \label{lem:reghypo} There holds:
 \begin{enumerate}[leftmargin=0.7cm]
 \item
 For any $k \in \N$ and $\delta >0,$ there are two positive constants $C_{k,\delta},R_\delta >0$ such that
$\mathrm{supp}\left(\mathcal{A}^{(\delta)}g\right)\subset B(0,R_{\delta})$
and
	\begin{equation}\label{eq:Adelta}
	 {\|\mathcal{A}^{(\delta)}g\|_{\W^{k,2}_{v}(\R^{d})}} 
	 \leq C_{k,\delta}\|g\|_{L^{1}_{v}(\langle v \rangle)}, \qquad \forall \,  g \in L^{1}_v(\langle v \rangle).
	 \end{equation}
\item
 There exist $\delta_0$, $\re_0$, $\nu_0$ such that for all $\re \in (\re_0,1)$ and $\delta \in (0,\delta_0)$,
 \begin{align*}
 \mathcal{B}_{\re,\e}^{(\delta)} + \e^{-2}\nu_0  \quad  \text{ is hypo--dissipative in $\E$ and } \Y_j, \, \, j =-1,0,1,
 \end{align*}
 where we recall that the spaces $\E$ and $\Y_j$ are respectively defined in~\eqref{def:E-E1} and~\eqref{def:Yj}. 
 \end{enumerate}
 \end{lem}

In what follows, we suppose that Assumption~\ref{hyp:re} is satisfied. We introduce $\e_0$ which is such that $\re(\e_0) = \re_0$ (and thus $\re(\e) \in (\re_0,1)$ for all $\e \in (0,\e_0)$) and consider $\delta \in (0,\delta_0)$, $\e \in (0,\e_0)$.
We will denote $\G_\e := \G_{\re,\e}$ as well as $\A_\e := \A_\e^{(\delta)}$ and $\B_\e := \B^{(\delta)}_{\alpha,\e}$ but do not change the notations $\G_{1,\e}$ and $\B_{1,\e}$. The following corollary states immediate consequences of the previous lemma (we denote by~$\Rs(\cdot,\B_{\e})$ the resolvent of the operator~$\B_\e$):

\begin{cor}\label{cor:reghypo}
There holds:
\begin{enumerate}[leftmargin=0.7cm]
\item For any $i,j \in \{-1,0,1\}$, we have
	$$
	\|\mathcal{A}_\e\|_{\Y_{i}\to \Y_{j}} %+ \|\mathcal{A}_\e\|_{\mathscr{B}(\E,\Y_{i})} 
	\lesssim \frac{1}{\e^2}.
	$$
\item If $\nu>0$ is fixed, then for $\e$ small enough (in terms of $\nu_0$ and $\nu$) and $j=-1,0,1$, 
	$$
	%\|\Rs(\lambda,\B_{\e})\|_{\mathscr{B}(\E)}+  
	\|\Rs(\lambda,\B_{\e})\|_{\Y_j\to \Y_{j}} \lesssim \frac{1}{\mathrm{Re}\,\lambda+\e^{-2}{\nu_0}}\lesssim \e^2, \qquad \forall \, \mathrm{Re}\,\lambda >-\nu\,.
	$$
\end{enumerate}
\end{cor}
\smallskip
The second keypoint to develop our perturbative argument is to have a good understanding of the spectrum of the operator $\G_{1,\e}$. We here give some estimates on the associated resolvent that are a consequence of a result of decay of the associated semigroup (see\cite[Theorem~2.12]{ALT} which gives an improved version of~\cite[Theorem 2.1]{bmam}):

\begin{lem}\label{prop:resG1e} 
There exists $\e_1 \in (0,\e_0)$ such that for $j=-1,0,1$, for any $\lambda \in \C_{\mu_{\star}}^{\star}$ and any~$\e \in (0,\e_1)$,
	$$
	%\|\Rs(\lambda,\G_{1,\e})\|_{\mathscr{B}(\E)} + 
	\|\Rs(\lambda,\G_{1,\e})\|_{\Y_j\to \Y_{j}}
	\lesssim \max\bigg(\frac{1}{|\lambda|},\frac{1}{\mathrm{Re}\lambda+\mu_{\star}}\bigg)
	$$
where $\mu_\star$ has been defined in~\eqref{eq:sgG1eps}. 
\end{lem}

\smallskip
Let us now explain how we develop our perturbative argument to prove Theorem~\ref{theo:linear}. The following proposition (which is an adaptation of~\cite[Lemma 2.16]{Tr}) is the first step in the development of the perturbative argument and its proof relies on Corollary~\ref{cor:reghypo} and Lemma~\ref{prop:resG1e}.
\begin{prop}\label{lem:inverse} For all $\lambda \in \C_{\mu_{\star}}^{\star}$, let
	$$
	\mathcal{J}_{\e}(\lambda)=\left(\G_{\e} -\G_{1,\e}\right)\Rs(\lambda,\G_{1,\e})\A_{\e}\,\Rs(\lambda,\B_{\e}).
	$$
Then, for any $\mu \in (0,\mu_{\star})$ and $\lambda \in \C_{\mu}\setminus \mathbb{D}(\mu_{\star}-\mu)$, there exists $\e_2 \in (0,\e_1)$ such that for any~$\e \in (0,\e_2)$, 
	\begin{equation}\label{eq:Jalk}
	\left\|\mathcal{J}_{\e}(\lambda)\right\|_{\Y\to\Y} \lesssim \frac{1}{\mu_\star-\mu}\,\frac{1-\re(\e)}{\e^{2}}.
	\end{equation}

\smallskip
\noindent
In addition, there exists $\e_3 \in (0,\e_2)$ and $\lambda_3>0$ such that for $\e \in (0,\e_3)$ and $\lambda_0 \in [0,\lambda_3)$ (where $\lambda_0$ is defined in Assumption~\ref{hyp:re}), $\mathbf{Id}-\mathcal{J}_{\e}(\lambda)$ and $\lambda-\G_{\e}$ are invertible in $\Y$ with
	\begin{equation}\label{eq:reso}
	\Rs(\lambda,\G_{\e})=\Gamma_{\e}(\lambda)(\mathbf{Id}-\mathcal{J}_{\e}(\lambda))^{-1}, \qquad \lambda \in \C_{\mu} \setminus \mathbb{D}(\mu_{\star}-\mu),
	\end{equation}
where $\Gamma_{\e}(\lambda):=\Rs(\lambda,\B_{\e})+\Rs(\lambda,\G_{1,\e})\A_{\e}\,\Rs(\lambda,\B_{\e})$.  Finally, we have for $\e \in (0,\e_3)$, 
	\begin{equation}\label{eq:estimR}
	\|\Rs(\lambda,\G_{\e})\|_{\Y\to\Y} \lesssim \frac{1}{\mu_\star-\mu}, \qquad  \lambda \in \C_{\mu} \setminus \mathbb{D}(\mu_{\star}-\mu).
	\end{equation}
\end{prop}

\noindent {\it Sketch of the proof.} The estimate on $\mathcal{J}_{\e}(\lambda)$ can be easily deduced from~\eqref{eq:perturb}, Corollary~\ref{cor:reghypo} and Lemma~\ref{prop:resG1e}. First, fix $\mu \in (0,\mu_\star)$ and notice that from Corollary~\ref{cor:reghypo}, we clearly have that there exists $\e_2 \in (0,\e_1)$ (which depends on $\nu_0$ and $\mu$) such that for any $\mathrm{Re}\,\lambda >-\mu$, we have:
	$$
	\left\|\A_{\e}\,\Rs(\lambda,\mathcal{B}_{\e})\right\|_{\Y\to\Y_1} \lesssim 1. 
	$$ 
We can then deduce that for $\mu \in (0,\mu_{\star})$, for any $\mathrm{Re}\,\lambda >-\mu$, $|\lambda| \geq \mu_\star-\mu$,
	\begin{equation} \label{eq:Jel}
	\begin{split}
	\left\|\mathcal{J}_{\e}(\lambda)\right\|_{\Y\to\Y} 
	&\leq \left\|\G_{\e}-\G_{1,\e}\right\|_{\Y_{1}\to\Y}\,\|\Rs(\lambda,\G_{1,\e})\|_{\Y_{1}\to\Y_{1}}\,\left\|\A_{\e}\,\Rs(\lambda,\mathcal{B}_{\e})\right\|_{\Y\to\Y_{1}}\\
	&\leq C \, \frac{1-\re(\e)}{\e^{2}}\,\,\max\left(\frac{1}{|\lambda|},\,\frac{1}{\mathrm{Re}\lambda+\mu_{\star}}\right) \leq C  \,\frac{1-\re(\e)}{\e^{2}} \frac{1}{\mu_\star-\mu}
	\end{split}
	\end{equation}
for some $C>0$. Moreover, one can choose $\e_3 \in (0,\e_2)$ and $\lambda_3 >0$ depending on the difference $\chi = \mu_{\star}-\mu$, so that if $\lambda_0 \in [0,\lambda_3)$ (recall that $\lambda_0$ is defined in Assumption~\ref{hyp:re} and is such that $(1-\alpha(\e)) \e^{-2} \sim \lambda_0 + \eta(\e)$)
	\begin{equation}\label{eq:def-re}
	\rho(\e):=\frac{C}{\mu_\star-\mu}\,\frac{1-\re(\e)}{\e^{2}} < 1, 
	\qquad \forall \, \e \in (0,\e_3).
	\end{equation}
Under such an assumption, one sees that, for all $\lambda \in \C_{\mu}\setminus \mathbb{D}(\mu_{\star}-\mu)$,  $\mathbf{Id}-\mathcal{J}_{\e}(\lambda)$ is invertible in $\Y$  with 
	$$
	(\mathbf{Id}-\mathcal{J}_{\e}(\lambda))^{-1}=\sum_{p=0}^{\infty}\left[\mathcal{J}_{\e}(\lambda)\right]^{p}, 
	\qquad \forall \, \e \in (0,\e_3).
	$$
Let us fix then $\e \in (0,\e_3)$ and $\lambda \in \C_{\mu}\setminus \mathbb{D}(\mu_{\star}-\mu)$. The range of $\Gamma_{\e}(\lambda)$ is clearly included in $\D(\B_{\e})=\D(\G_{1,\e})$. Then, writing $\G_{\e}=\A_{\e}+\B_{\e}$, we easily get that
	$$
	(\lambda-\G_{\e})\Gamma_{\e}(\lambda)=\mathbf{Id}-\mathcal{J}_{\e}(\lambda)
	$$
i.e. $\Gamma_{\e}(\lambda)(\mathbf{Id}-\mathcal{J}_{\e}(\lambda))^{-1}$ is a right-inverse of $(\lambda-\G_{\e}).$ To prove that $\lambda-\G_{\e}$ is invertible, it is therefore enough to prove that it is one-to-one, which can be done up to reducing the value of $\e_3$ using~\eqref{eq:perturb}, Lemmas~\ref{lem:reghypo} and~\ref{prop:resG1e}. Thus, for $\e\in (0,\e_3)$, $\C_{\mu} \setminus \mathbb{D}(\mu_{\star}-\mu)$ is included into the resolvent set of $\G_{\e}$ and this shows \eqref{eq:reso}. To estimate now $\|\Rs(\lambda,\G_{\e})\|_{\mathscr{B}(\Y)}$, one simply notices that
	\begin{equation}\label{eq:I-Jel}
	\|(\mathbf{Id}-\mathcal{J}_{\e}(\lambda))^{-1}\|_{\Y\to\Y} \leq \sum_{p=0}^{\infty}\|\mathcal{J}_{\e}(\lambda)\|_{\Y\to\Y}^{p}\leq \frac{1}{1-\rho(\e)}, \quad \forall\, \lambda \in \C_{\mu}\setminus \mathbb{D}(\mu_{\star}-\mu)
	\end{equation}
from which, as soon as $\lambda \in \C_{\mu}\setminus\mathbb{D}(\mu_{\star}-\mu)$,
	$$
	\|\Rs(\lambda,\G_{\e})\|_{\Y\to\Y}\leq \frac{1}{1-\rho(\e)}\,\|\Gamma_{\e}(\lambda)\|_{\Y\to\Y}\,.
	$$
One checks, using the previous computations, that for $\lambda\in \C_{\mu}\setminus\mathbb{D}(\mu_{\star}-\mu)$,
	\begin{equation}\label{eq:normGe}
	\|\Gamma_{\e}(\lambda)\|_{\Y\to\Y} \lesssim \e^{2}+\|\A\|_{\Y\to\Y}\|\Rs(\lambda,\G_{1,\e})\|_{\Y\to\Y}
	\end{equation}
and deduces \eqref{eq:estimR}.  This achieves the proof.
\qed

A first obvious consequence of Proposition \ref{lem:inverse} is that, for any $\mu \in (0,\mu_{\star})$, there is $\e_3>0$ depending only on $\chi=\mu_{\star}-\mu$ such that
	$$
	\mathfrak{S}(\G_{\e}) \cap \C_\mu \subset \mathbb{D}(\mu_\star-\mu), \qquad \forall\, \e \in (0,\e_3).
	$$
We denote by $\mathbf{P}_{\e}$ (resp. $\mathbf{P}_0$) the spectral projection associated to the set
	$$
	\mathfrak{S}(\G_{\e}) \cap \C_{\mu}=\mathfrak{S}(\G_{\e}) \cap  \mathbb{D}(\mu_{\star}-\mu)
	\quad (\text{resp.} \quad \mathfrak{S}(\G_{1,\e}) \cap \C_{\mu} = \{0\}).
	$$

One can deduce then the following lemma whose proof is similar to \cite[Lemma 2.17]{Tr}. 
\begin{lem}\phantomsection\label{lem:PaP0} 
For any $\mu \in (0,\mu_\star)$ such that $\C_\mu \subset \mathbb{D}(\mu_\star-\mu)$, there exist $\e_4 \in (0,\e_3)$ and $\lambda_4 \in (0,\lambda_3)$ depending only on $\chi=\mu_{\star}-\mu$ such that if $\lambda_0 \in [0,\lambda_4)$ (where $\lambda_0$ is defined in Assumption~\ref{hyp:re}), 
	$$
	\left\|\mathbf{P}_{\e}-\mathbf{P}_{0}\right\|_{\Y\to\Y} < 1, \qquad \forall\, \e \in (0,\e_4).
	$$
In particular,
	\begin{equation}\label{eq:dim}
	\mathrm{dim}\,\mathrm{Range}(\mathbf{P}_{\e})=\mathrm{dim}\,\mathrm{Range}(\mathbf{P}_{0})=d+2, \qquad \forall\, \e \in (0,\e_4).
	\end{equation}
\end{lem}
\noindent {\it Sketch of the proof.}  Let $\mu \in (0,\mu_\star)$ be close enough to $\mu_\star$ so that $\mathbb{D}(\mu_\star - \mu) \subset \C_\mu$ and $0 < r < \chi=\mu_{\star}-\mu$. One has $\mathbb{D}(r) \subset \C_{\mu}^{\star}$. We set $\gamma_{r}:=\{z \in \C\;;\;|z|=r\}$. Recall that by definition
	$$
	\mathbf{P}_{\e}:=\frac{1}{2i\pi}\oint_{\gamma_{r}}\Rs(\lambda,\G_{\e})\, \d\lambda, \qquad \mathbf{P}_{0}:=\frac{1}{2i\pi}\oint_{\gamma_{r}}\Rs(\lambda,\G_{1,\e})\, \d\lambda.
	$$
For $\lambda \in \gamma_{r}$, set
	$$
	\mathcal{Z}_{\e}(\lambda)=\Rs(\lambda,\G_{1,\e})\A_{\e}\Rs(\lambda,\B_{\e})
	$$
so that $\Gamma_{\e}(\lambda)=\Rs(\lambda,\B_{\e})+\mathcal{Z}_{\e}(\lambda)$. Recall from \eqref{eq:reso} that, for $\lambda \in \gamma_{r}$,
	\begin{multline*}
	\Rs(\lambda,\G_{\e})=\Rs(\lambda,\B_{\e})(\mathbf{Id}-\mathcal{J}_{\e}(\lambda))^{-1}+\mathcal{Z}_{\e}(\lambda)(\mathbf{Id}-\mathcal{J}_{\e}(\lambda))^{-1}\\
	=\Rs(\lambda,\B_{\e})+\Rs(\lambda,\B_{\e})\mathcal{J}_{\e}(\lambda)(\mathbf{Id}-\mathcal{J}_{\e}(\lambda))^{-1}+\mathcal{Z}_{\e}(\lambda)(\mathbf{Id}-\mathcal{J}_{\e}(\lambda))^{-1}
	\end{multline*}
where we wrote $(\mathbf{Id}-\mathcal{J}_{\e}(\lambda))^{-1}=\mathbf{Id}+\mathcal{J}_{\e}(\lambda)(\mathbf{Id}-\mathcal{J}_{\e}(\lambda))^{-1}$ to get the second equality. One also has
	$$
	\Rs(\lambda,\G_{1,\e})
	=\Rs(\lambda,\B_{1,\e})+\Rs(\lambda,\G_{1,\e})\A_{\e}\left[\Rs(\lambda,\B_{1,\e})-\Rs(\lambda,\B_{\e})\right]+\mathcal{Z}_{\e}(\lambda).
	$$
One can then obtain (see the proof of~\cite[Lemma~3.8]{ALT} for the details)
	\begin{equation*}
	\mathbf{P}_{\e}-\mathbf{P}_{0}=\frac{1}{2i\pi}\oint_{\gamma_{r}}\Gamma_{\e}(\lambda)\mathcal{J}_{\e}(\lambda)(\mathbf{Id}-\mathcal{J}_{\e}(\lambda))^{-1}\, \d\lambda
	+\frac{1}{2i\pi}\oint_{\gamma_{r}}\Rs(\lambda,\G_{1,\e})\A_{\e}\left[\Rs(\lambda,\B_{\e})-\Rs(\lambda,\B_{1,\e})\right]\, \d\lambda.
	\end{equation*}
The first part is estimated thanks to~\eqref{eq:Jel},~\eqref{eq:I-Jel} and~\eqref{eq:normGe} combined with Lemma~\ref{prop:resG1e}: 
	$$
	\|\Gamma_{\e}(\lambda)\mathcal{J}_{\e}(\lambda)(\mathbf{Id}-\mathcal{J}_{\e}(\lambda))^{-1}\|_{\Y\to\Y}
	\lesssim \frac{1}{r^2} \frac{1}{1-\rho(\e)} \frac{1-\re(\e)}{\e^2}. 
	$$
For the second part, notice first that from Lemma~\ref{prop:resG1e},
	\begin{equation*}
	\left\|\Rs(\lambda,\G_{1,\e})\A_{\e}\left[\Rs(\lambda,\B_{\e})-\Rs(\lambda,\B_{1,\e})\right]\right\|_{\Y\to\Y} 
	\lesssim \frac{1}{r}\,\left\|\A_{\e}\Rs(\lambda,\B_{\e})-\A_{\e}\Rs(\lambda,\B_{1,\e})\right\|_{\Y\to\Y}.
	\end{equation*}
Then, for $\lambda \in \gamma_{r}$, we have 
	$$
	\A_{\e}\Rs(\lambda,\B_{\e})-\A_{\e}\Rs(\lambda,\B_{1,\e})=\A_{\e}\Rs(\lambda,\B_{\e})\left[\B_{\e}-\B_{1,\e}\right]\Rs(\lambda,\B_{1,\e})
	$$
which implies that
	\begin{multline*}
	\left\|\A_{\e}\Rs(\lambda,\B_{\e})-\A_{\e}\Rs(\lambda,\B_{1,\e})\right\|_{\Y\to\Y}
	\leq \|\A_{\e}\Rs(\lambda,\B_{\e})\|_{\Y_{-1}\to\Y}\,\|\B_{\e}-\B_{1,\e}\|_{\Y\to\Y_{-1}}\,\|\Rs(\lambda,\B_{1,\e})\|_{\Y\to\Y} 
	\lesssim \frac{1-\re(\e)}{\e^2}. 
	\end{multline*}
Proceeding as in the proof of Lemma~\ref{lem:inverse}, one can 
conclude that for any $0<r<\chi=\mu_{\star}-\mu$,
	\begin{equation}\label{eq:PePo}
	\|\mathbf{P}_{\e}-\mathbf{P}_{0}\|_{\Y\to\Y} 
	\leq \frac{C}{r} \,\frac{1-\re(\e)}{\e^{2}} \left(\frac{1}{r (1-\rho(\e))} +1 \right):=\ell(\e).
	\end{equation}
Thanks to Assumption \ref{hyp:re}, one can find $\e_4$ and $\lambda_4$ depending only on $\chi$ such that $\ell(\e) < 1$ for any $\e \in (0,\e_4)$ and $\lambda_0 \in [0,\lambda_4)$. In particular, we deduce \eqref{eq:dim} from \cite[Paragraph~I.4.6]{kato}.
\qed

\medskip
With Lemma \ref{lem:PaP0}, we can now end the proof of Theorem \ref{theo:linear}.

\smallskip
\noindent {\it Sketch of the proof of Theorem \ref{theo:linear}.}
The structure of $\mathfrak{S}(\G_{\e}) \cap \C_{\mu}$ in the space $\Y$ comes directly from Lemma \ref{lem:PaP0} together with Proposition \ref{lem:inverse}. To describe more precisely the spectrum, one first remarks that
$$
\mathfrak{S}(\mathscr{L}_{\re(\e)}) \cap \C_\mu \subset \mathfrak{S}(\G_{\e})  \cap \C_\mu.
$$
This comes from the fact that for each eigenvalue of $\mathscr{L}_{\re(\e)}$, the eigenfunction depends only on $v$ and thus remains an eigenfunction for the operator $\G_\e$. 
Since, for $\e$ small enough, the same perturbative argument that we developed above implies that the spectral projection~$\Pi_{\mathscr{L}_{\re(\e)}}$ associated to $\mathfrak{S}(\mathscr{L}_\re) \cap \C_{\mu}$ satisfies
$$\mathrm{dim(Range}(\Pi_{\mathscr{L}_{\re(\e)}}))=\mathrm{dim(Range}(\Pi_{\mathscr{L}_{1}}))=d+2=\mathrm{dim(Range}(\mathbf{P}_{\e})),$$
we get that
\begin{equation}\label{eq:spectr=}
\mathfrak{S}(\mathscr{L}_{\re(\e)}) \cap \C_{\mu} = \mathfrak{S}(\G_{\e})  \cap \C_{\mu}\,,
\end{equation}
that is, the eigenvalues $\lambda_{j}(\e)$ are actually eigenvalues of~$\mathscr{L}_{\re(\e)}$. The development of the energy eigenvalue $\lambda_{d+2}(\e)$ comes from~\cite{MiMo3}. The conservation of mass gives us that $0$ is an eigenvalue for our problem. The intermediate eigenvalues $\lambda_j(\e)$ for $j=2, \dots, d+1$ are obtained thanks to the fact that
$$
\int_{ \R^{d}}\mathscr{L}_{\re(\e)} \varphi(v)\,v_{i}\,\d v =-\frac{1-\re(\e)}{\e^2} \int_{\R^{d}}v_{i}\nabla \cdot (v\varphi(v))\,\d v=\frac{1-\re(\e)}{\e^2}\int_{\R^{d}}v_{i}\,\varphi(v)\,\d v.
$$
Notice that all this allows us to find eigenfunctions (that depend only on $v$) in $L^2_{v,x}(\langle v \rangle^r)$. Using once more the splitting $\mathscr{L}_{\re}=\A^{(\delta)}+\B_{\alpha}^{\delta}$ defined in~\eqref{eq:splitLalpha} and the regularizing properties of $\A^{(\delta)}$, one can actually prove that our eigenfunctions lie in $\Y$, which yields the conclusion of Theorem~\ref{theo:linear} in the space $\Y$. To extend the result to the space $\E$, we use an enlargement argument coming from~\cite{GMM}, we omit the details here and just mention that this argument is based on the splitting $\G_\e = \A_\e+\B_\e$ introduced in~\eqref{eq:splitGeps}.  
\qed

%%%%%%%%%%%%%%%%%%%%%%%%%%%%%%%%%%%%%%%%%%%%%%%%%%%%%%%%%%%
\section{Study of the kinetic nonlinear problem} \label{sec:nonlinear}

Let us recall that the spaces~$\E$ and~$\E_1$ are defined in~\eqref{def:E-E1}. In this section, we assume that Assumption~\ref{hyp:re} is met and consider~$\e \in (0,\overline\e)$, $\lambda_0 \in \big[0,\overline \lambda\big]$ where~$\overline\e$ and $\overline \lambda$ are defined in Theorem~\ref{theo:linear}. As in Section~\ref{sec:linear}, to lighten the notations, we write~$\G_\e = \G_{\re(\e),\e}$ as well as~$\B_\e=\B_{\re(\e),\e}$. 

\subsection{Splitting of the nonlinear inelastic Boltzmann equation}
Now that the spectral analysis of the linearized operator~$\G_\e$ in the space $\E$ has been performed, in order to prove Theorem~\ref{theo:main-cauchy}, we are going to prove several \emph{a priori estimates} for the solutions to \eqref{eq:BE}. The crucial point in the analysis lies in the splitting of~\eqref{eq:BE} into a system of two equations mimicking a spectral enlargement method from a PDE perspective {(see \cite[Section 2.3]{MiMoFP} and \cite{bmam} for pioneering ideas on such a method)}. 
More precisely, using~\eqref{eq:splitGeps}, the splitting amounts to look for a solution of \eqref{eq:BE} of the form 
	$$
	h_{\e}(t)=\ho_{\e}(t)+\hu_{\e}(t)
	$$
with $\ho_{\e}$ solution to 
	\begin{equation}\label{eq:h0}
	\hspace{-.4cm}\left\{\begin{array}{ccl}
	\partial_{t} \ho_\e&=&\!\!\!\B_{\e}\ho_\e + \frac1\e\Q_{\re(\e)}(\ho_\e,\ho_\e) + \frac1\e\Big[\Q_{{\re(\e)}}(\ho_\e,\hu_\e)+\Q_{{\re(\e)}}(\hu_\e,\ho_\e)\Big] \\ [10pt]
	&& + {\Big[\G_{\e}\hu_\e-\G_{1,\e}\hu_\e\Big] + \frac1\e\Big[\Q_{{\re(\e)}}(\hu_\e,\hu_\e)-\Q_{1}(\hu_\e,\hu_\e)\Big]},\\ [10pt]
	\ho_\e(0,x,v)&=&\!\!\!h_{\e}^{\mathrm{\mathrm{in}}}(x,v) \in \E,
	\end{array}\right.
	\end{equation}
and $\hu_\e$ solution to 
	\begin{equation}\label{eq:h1}
	\left\{
	\begin{array}{ccl}
	\partial_{t} \hu_\e&=& {\G_{1,\e}\hu_\e} + \frac1\e\Q_{1}(\hu_\e,\hu_\e) + \A_\e \ho_\e,\\[10pt]
	\hu_\e(0,x,v)&=&0.
	\end{array}\right.
	\end{equation}
In order to lighten the notations, in this section, we will write~$h^{\mathrm{in}}$,~$h$,~$\ho$ and~$\hu$ instead of~$h^{\mathrm{in}}_\e$,~$h_\e$,~$\ho_\e$ and~$\hu_\e$. 
The goal is to obtain nice nested a priori estimates on $\ho$ and~$\hu$. Notice first that our splitting is more complicated than the one of~\cite{bmam} because it relies on perturbative considerations around the elastic case that come out in the equation satisfied by $\ho$. As a consequence, our a priori estimates are more intricate and require the use of non standard Gronwall lemma. Notice also that since the initial datum of~$\hu$ is vanishing, we can study the equation on~$\hu$ in any functional space. In particular, we can study it in the Hilbert space~$\H = \W_{x,v}^{m,2}\left(\M^{-1/2}\right)$
in which we have a good understanding of the elastic linearized operator~$\G_{1,\e}$. Indeed, in this type of spaces, the symmetries of the collision operator $\Q_1$ allow to get some nice hypocoercive estimates (see~\eqref{eq:hypocoerc}).

\begin{nb}
In~\cite{GMM}, the authors treat the elastic case ($\alpha=1$) of the non-rescaled equation~($\e=1$) and they do not resort to such a splitting method to study the nonlinear equation, their approach is based on the use of a norm which is equivalent to the usual one and is such that~$\G_{1,1}$ is dissipative in this norm in large spaces. Such an approach is no longer usable when one wants to deal with rescaled equations and obtain uniform in $\e$ estimates. Indeed, the definition of the equivalent norm in~\cite{GMM} does not take into account the different behaviors of microscopic and macroscopic parts of the solution with respect to~$\e$: typically, the microscopic part of the solution vanishes as $\e \to 0$ whereas the macroscopic one does not. Conversely, in the splitting method, the equation that defines $\hu$ is treated thanks to hypocoercivity tricks that allow to distinguish microscopic and macroscopic behaviors. 
\end{nb}

\subsection{Estimating $\ho$}
Concerning $\ho$, let us first mention that the dissipativity properties of~$\B_{\e}$ stated in Lemma~\ref{lem:reghypo} can actually be improved a bit. More precisely, one can show that there exist norms on the spaces $\E$ and $\E_1$ that are equivalent to the standard ones (with multiplicative constants independent of $\e$) that we still denote $\|\cdot\|_{\E}$ and $\|\cdot\|_{\E_1}$ and that satisfy:
	\begin{equation} \label{eq:dissipE1}
	\frac{\d}{\d t}\|S_{\B_{\e}}(t)g\|_{\E} \leq -\frac{\nu_{0}}{\e^2}\|S_{\B_{\e}}(t)g\|_{\E_{1}}
	\end{equation}
where we have denoted by $\left(S_{\B_{\e}}(t)\right)_{t\geq0}$ the semigroup generated by~$\B_\e$ and $\nu_0$ is defined in Lemma~\ref{lem:reghypo}. 
Let us also introduce the Banach space $\mathcal{E}_2$
	$$
	\E_{2}:={\W^{k+1,2}_{v}\W^{m,2}_{x}(\m_{q+{2\kappa}+2}), \quad \kappa > \frac{d}{2}}
	$$ 
which satisfies the following continuous embeddings: $\H \hookrightarrow \mathcal{E}_2 \hookrightarrow \mathcal{E}_1$  (recall that $\E_1$ is defined in~\eqref{def:E-E1}). Let us point out that the spaces $\E_1$ and $\E_2$ allow us to get the following estimates (see~\cite[Remark~3.5]{ALT} and~\cite{ACG,AG}):
	\begin{equation} \label{eq:bilinear}
	\|(\Q_\re-\Q_1)(g,f)\|_\E \lesssim (1-\re) \|g\|_{\E_2} \|f\|_{\mathcal{E}_2} \quad
	 \text{and} \quad 
	\|\Q_\re(g,f)\|_{\mathcal{E}} \lesssim \|g\|_\E \|f\|_{\E_1} + \|g\|_{\E_1} \|f\|_\E
	\end{equation}
where the multiplicative constants are uniform in $\re$.
One can then obtain the following proposition: 
\begin{prop}\label{prop:h0} 
Assume that {$\ho\in \E$, {$\hu \in \H$}} are such that
	\begin{equation*}
	\sup_{t\geq0}\big(\|\ho (t)\|_{\E}  + \|\hu(t)\|_{{{\H}}} \big) \leq \Delta_0 <\infty.
	\end{equation*}
For $\nu \in (0,\nu_0)$ (where $\nu_0$ is defined in Lemma~\ref{lem:reghypo}), there exists an explicit $\e_5 \in (0,\overline \e)$ (where~$\overline \e$ is defined in Theorem~\ref{theo:linear}) such that:
\begin{equation} \label{ineq:ho}
	\begin{split}
	\|\ho(t)\|_{\E} \lesssim 
	\|h^{\mathrm{in}}\|_{\E}\,e^{-\frac{\nu}{\e^2}t} +  {\lambda}_{\e} \int^{t}_{0}e^{-\frac{\nu}{\e^2}(t-s)}\|\hu(s)\|_{\H}\,\d s
	%+\e\, {\lambda}_{\e} \int^{t}_{0}e^{-\frac{\nu}{\e^2}(t-s)}\|\hu(s)\|^2_{\H}\,\d s
	\end{split}
\end{equation}
where we recall that $\lambda_\e \underset{\e\to 0}{\sim} \frac{1-\re(\e)}{\e^2}$ is defined in Theorem~\ref{theo:linear}. 
\end{prop}

\noindent {\it Sketch of the proof.} 
Using~\eqref{eq:dissipE1} as well as~\eqref{eq:bilinear} and recalling that $\ho$ solves~\eqref{eq:h0}, we can compute the evolution of $\|\ho(t)\|_\E$ and estimate it:
	\begin{multline} \label{eq:inegho}
	\frac{\d}{\d t} \|\ho(t)\|_{\mathcal{E}} 
	\leq -\frac{\nu_0}{\e^2} \|\ho(t)\|_{\mathcal{E}_1} + \frac{C}{\e} \left(\|\ho(t)\|_{\mathcal{E}} + \|\hu(t)\|_{\mathcal{E}_1}\right) \|\ho(t)\|_{\mathcal{E}_1} \\
	+ C \, \frac{1-\re(\e)}{\e^2} \|\hu(t)\|_{\mathcal{E}_2} + C \, \frac{1-\re(\e)}{\e} \|\hu(t)\|^2_{\mathcal{E}_2}.
	\end{multline}
Using that the embedding $\E_2 \hookrightarrow \H$ is continuous, recalling that $h^0(0)=h^{\mathrm{in}}$ and choosing~$\e_5$ small enough so that~$C \, \e_5 \, \Delta_0 \leq \nu_0-\nu$, we obtain 
\begin{equation*}
	\begin{split}
	\|\ho(t)\|_{\E} \lesssim 
	\|h^{\mathrm{in}}\|_{\E}\,e^{-\frac{\nu}{\e^2}t} +  {\lambda}_{\e} \int^{t}_{0}e^{-\frac{\nu}{\e^2}(t-s)}\|\hu(s)\|_{\H}\,\d s
	+\e\, {\lambda}_{\e} \int^{t}_{0}e^{-\frac{\nu}{\e^2}(t-s)}\|\hu(s)\|^2_{\H}\,\d s.
	\end{split}
\end{equation*}
We conclude to~\eqref{ineq:ho} by assuming furthermore that $\e_5 \Delta_0 \leq 1$.  
\qed 

\subsection{Estimating $\hu$}
We now comment and study the equation satisfied by~$\hu$. 
Let us point out that getting estimates on $\hu$ is trickier than in~\cite{bmam}, indeed, in the latter paper, the idea is to estimate separately $\mathbf{P}_0 h^1$ and~$(\mathbf{Id} - \mathbf{P}_0) \hu$ where $\mathbf{P}_{0}$ is the projector onto $\operatorname{Ker} (\G_{1,\e})$ defined by 
	\begin{equation}\label{eq:P0}
	\mathbf{P}_{0}g:=\sum_{i=1}^{d+2}\left(\int_{\T^{d}\times\R^{d}}g\,\Psi_{i}\, \d v\,\d x \right)\,\Psi_{i}\,\M
	\end{equation}
where the functions $\Psi_i$ have been defined in~\eqref{def:Psii},
and thanks to the properties of preservation of mass, momentum and energy of the whole equation, one could write that $\mathbf{P}_{0} h =0$ so that $\mathbf{P}_{0} \hu = - \mathbf{P}_{0} \ho$ and directly get an estimate on~$\mathbf{P}_{0} \hu$ from the one on~$\ho$. In our case, the energy is no longer preserved which induces additional difficulties. However, we keep the same strategy and start by estimating~$\mathbf{P}_0 \hu$ (see Remark~\ref{nb:BMM} for a comment on this choice of strategy). 

For the sequel, we also introduce 
	\begin{equation}\label{eq:PP0}
	\mathbb{P}_{0}h=\sum_{i=1}^{d+1}\left(\int_{\T^{d}\times\R^{d}}h\,\Psi_{i}\,\d v\, \d x\right)\,\Psi_{i}\,\M\,, 
	\quad 
	\Pi_{0}h=\left(\int_{\T^{d}\times\R^{d}}h\Psi_{d+2}\,\d v\, \d x\right)\,\Psi_{d+2}\,\M.
	\end{equation}
Notice that thanks to Cauchy-Schwarz inequality in velocity, one can easily prove that we have $\mathbb{P}_0 \in \mathscr{B}(\E,\H)$. 
One can then obtain the following proposition:
\begin{prop} \label{prop:P0h1}
Assume that {$\ho\in \E$, {$\hu \in \H$}} are such that
	\begin{equation*}
	\sup_{t\geq0}\big(\|\ho (t)\|_{\E}  + \|\hu(t)\|_{{{\H}}} \big) \leq \Delta_0 <\infty.
	\end{equation*}
For $\e \in (0,\e_5)$ ($\e_5$ is defined in Proposition~\ref{prop:h0}), 
	\begin{multline} \label{eq:P0h1}
	\|\mathbf{P}_0 h^1(t)\|_\E \lesssim
	\|h^0(t)\|_\E + \|h^{\mathrm{in}}\|_\E \, e^{-\overline \lambda_\e t} 
	+ \lambda_\e \int_0^t e^{-\overline \lambda_\e(t-s)} \left(\|\ho(s)\|_\E + \|(\mathbf{Id} - \mathbf{P}_0)\hu(s)\|_\H \right)  \, \d s \\
	+ \e \lambda_\e \int_0^t e^{-\overline \lambda_\e(t-s)}  \|\hu(s)\|_\H  \, \d s 
	\end{multline}
where $\overline \lambda_\e = \lambda_\e + \operatorname{O}(1-\re(\e))$ with $\lambda_\e \underset{\e\to 0}{\sim} \frac{1-\re(\e)}{\e^2}$ defined in Theorem~\ref{theo:linear}. 
\end{prop}

\noindent {\it Sketch of the proof.} 
Due to the properties of preservation of mass and vanishing momentum of our equation, we have $\mathbb{P}_0 h = 0$ which implies that $\mathbb{P}_0 \hu = - \mathbb{P}_0 \ho$. Consequently, we easily get an estimate on $\mathbb{P}_0 \hu$ using that $\mathbb{P}_0 \in \mathscr{B}(\E,\H)$: 
	\begin{equation} \label{eq:estimP0h1}
	\|\mathbb{P}_0 \hu(t)\|_\H \lesssim \|\ho(t)\|_\E.
	\end{equation}
It now remains to estimate $\Pi_0 \hu$. To this end, we first notice that 
	$$
	\Pi_0 \hu = \mathbf{P}_0 \hu - \mathbb{P}_0 \hu = \mathbf{P}_0 h - \mathbf{P}_0 \ho - \mathbb{P}_0 \hu 
	= \Pi_0 h- \mathbf{P}_0 \ho - \mathbb{P}_0 \hu 
	$$
where we used that $\mathbf{P}_0 h = \Pi_0 h$ due to the preservation of mass and vanishing momentum so, using~\eqref{ineq:ho} and~\eqref{eq:estimP0h1}, we only need to estimate $\Pi_0 h$ to get an estimate on $\mathbf{P}_0 \hu$. To this end, we start by computing the evolution of $\Pi_0 h$:  
	$$
	\partial_t (\Pi_0 h) = \Pi_0 (\G_\e h) + \frac1\e \Pi_0 \Q_{\re(\e)} (h,h).
	$$
By direct inspection, using the definition of $\Pi_0$ given in~\eqref{eq:PP0} and the dissipation of energy~\eqref{eq:Dre} (see~\cite[Lemmas~4.2 and~4.5]{ALT}), we obtain: as $\e \to 0$,
	$$
	\Pi_0 (\G_\e h) 
	= -\overline \lambda_\e \Pi_0 h + \operatorname{O}\left(\frac{1-\re(\e)}{\e^2}\, \|(\mathbf{Id}- \mathbf{P}_0) h\|_{\E}\right)
	$$
with
	$$
	\overline \lambda_\e = \lambda_\e + \operatorname{O} (1-\alpha(\e)) \underset{\e \to 0}{\sim} \lambda_\e.
	$$
Similarly, we have by direct computation that
	$$
	|\Pi_0 \Q_{\re(\e)} (h,h)| = (1-\re^2) |\mathcal{D}_{\re(\e)} (h,h)| \Psi_{d+2} \, \M
	$$
so that, using Minkowski's inequality to estimate $\mathcal{D}_{\re(\e)} (h,h)$, we obtain
	\begin{equation} \label{eq:Pi0Qre}
	\|\Pi_0 \Q_{\re(\e)} (h,h)\|_\E \lesssim \e^2 \|h\|^2_\E. 
	\end{equation}
Gathering previous estimates, we are able to deduce that 
	\begin{multline*}
	\|\mathbf{P}_0 h^1(t)\|_\E \lesssim
	\|h^0(t)\|_\E + \|h^{\mathrm{in}}\|_\E e^{-\overline \lambda_\e t} 
	+ \lambda_\e \int_0^t e^{-\overline \lambda_\e(t-s)} \left(\|\ho(s)\|_\E + \|(\mathbf{Id}- \mathbf{P}_0) \hu(s)\|_\H \right)  \, \d s \\
	+ \e \lambda_\e \int_0^t e^{-\overline \lambda_\e(t-s)} \left(\|\ho(s)\|_\E^2 + \|\hu(s)\|^2_\H\right)  \, \d s. 
	\end{multline*}
With this, inequality~\eqref{eq:P0h1} holds by using $\e_5 \Delta_0 \leq 1$ from the proof of Proposition~\ref{prop:h0}. 
\qed

\begin{nb} \label{nb:BMM}
A natural approach would have been to adapt the method of~\cite{bmam} by applying~$\mathbf{P}_\e$ (the projector associated to the eigenvalues $\lambda_j(\e)$ for $j = 1, \dots, d+2$ of $\G_\e$ around $0$ that have been exhibited in Theorem~\ref{theo:linear}) to our equation instead of $\mathbf{P}_0$. It implies that one would have had to estimate $\Pi_\e h$ where $\Pi_\e$ is the projector associated to the energy eigenvalue $-\lambda_\e = \lambda_{d+2}(\e)$ defined in Theorem~\ref{theo:linear}. On the one hand, it simplifies the approach because $\Pi_\e \G_\e h = - \lambda_\e \Pi_\e h$ by definition. On the other hand, this projector is not explicit contrary to $\Pi_0$ and when applying $\Pi_{\e}$ to the equation satisfied by $h$
	$$
	\partial_{t} h= \G_{\e}h + \frac1\e\Q_{\re(\e)}(h,h),
	$$
nothing guarantees that $\Pi_{\e}\left[\e^{-1}\Q_{\re(\e)}(h,h)\right]$ remains of order $1$ with respect to $\e$ whereas we have seen in~\eqref{eq:Pi0Qre} that due to the dissipation of kinetic energy,  $\Pi_{0}\left[\e^{-1}\Q_{\re(\e)}(h,h)\right]$ is actually of order~$\e$. This explains our choice of strategy. 
\end{nb}

Let us now focus on the estimate of $(\mathbf{Id} - \mathbf{P}_0) \hu$. We can proceed similarly as in~\cite{bmam}, using in particular that $\mathbf{P}_0 \Q_1 = 0$. Another crucial point is that the source term $\mathcal{A}_\e \ho$ can be bounded in $\H$ using the fact that $\mathcal{A}_\e \in \mathscr{B}(\E,\H)$ (see Lemma~\ref{lem:reghypo}). Moreover, it is important to mention that the fact that the bound on~$\mathcal{A}_\e$ induces a rate of $\e^{-2}$ will be counterbalanced by the fact that the semigroup associated with $\B_{\e}$ has an exponential decay rate of type $e^{-\nu t/\e^2}$ (see~\eqref{eq:dissipE1}). We recall that the Hilbert space $\H_1$ is defined in~\eqref{def:H1} and is such that 
	\begin{equation} \label{eq:inegQ1H}
	\|\Q_1(g,g)\|_\H \lesssim \|g\|_\H \|g\|_{\H_1}. 
	\end{equation}
\begin{prop}\label{prop:Phi} 
Assume that {$\ho\in \E$, {$\hu \in \H$}} are such that
	\begin{equation*}
	\sup_{t\geq0}\big(\|\ho (t)\|_{\E}  + \|\hu(t)\|_{{{\H}}} \big) \leq \Delta_0 <\infty.
	\end{equation*}
For $\mu \in (0,\mu_\star)$ (where $\mu_\star$ is defined in~\eqref{eq:sgG1eps}) and for $\Delta_0$ small enough, we have that:
\begin{equation}
	\begin{split}\label{eq:estimPhi}
	\|(\mathbf{Id}-\mathbf{P}_0) h^1(t)\|^2_{\H} \lesssim 
	\Delta_0^2 \int_0^t e^{-\mu(t-s)} \|h^1(s)\|_\H^2 \, \d s 
	+ \frac{1}{\e^2} \int_0^t e^{-\mu(t-s)} \|h^1(s)\|_\H \|h^0(s)\|_\E \, \d s.
	\end{split}
\end{equation}
\end{prop}

\noindent {\it Sketch of the proof.} 
From~\eqref{eq:h1}, the fact that $\mathbf{P}_0 \Q_1(g,g)=0$ and the fact that $\mathbf{P}_0$ commutes with $\G_{1,\e}$, we can compute the evolution of $\Phi(t) := (\mathbf{Id} - \mathbf{P}_0) \hu$:
	$$
	\partial_t \Phi = \G_{1,\e} \Phi + \frac1\e \Q_1(h^1,h^1) + (\mathbf{Id} - \mathbf{P}_0) \A_\e h^0. 
	$$
We now use the hypocoercive norm on $\H$ for $\G_{1,\e}$ introduced in~\eqref{eq:hypocoerc} and also denote by $\Phi^\perp$ the microscopic part of~$\Phi$, namely $\Phi^\perp := (\mathbf{Id} - \bm{\pi}_0) \Phi$ where we recall that $\bm{\pi}_0$ is the projection onto the kernel of $\mathscr{L}_1$ that has been introduced in~\eqref{def:pi0}. We compute the evolution of $\|\Phi(t)\|^2_\H$: 
	$$
	\frac12 \frac{\d}{\d t} \|\Phi(t)\|^2 = \langle \G_{1,\e} \Phi(t),\Phi(t) \rangle_\H + \frac1\e \langle \Q_1(h^1(t),h^1(t)),\Phi^\perp(t) \rangle_\H + \langle (\mathbf{Id} - \mathbf{P}_0) \A_\e h^0(t), \Phi(t) \rangle_\H. 
	$$
Notice that we have been able to replace $\Phi$ by $\Phi^\perp$ in the second term due to the conservation laws satisfied by~$\Q_1$ and the fact that $\bm{\pi}_0$ is orthogonal in~$\H$. 
Then, from the properties of the hypocoercive norm (see~\eqref{eq:hypocoerc}), using~\eqref{eq:inegQ1H} and the facts that $\mathbf{P}_0 \in \mathscr{B}(\H)$, $\A \in \mathscr{B}(\E,\H)$ (from Lemma~\ref{lem:reghypo}) as well as Cauchy-Schwarz inequality, we obtain that 			
$$	\frac12 \frac{\d}{\d t} \|\Phi(t)\|_\H^2 \leq - \frac{\mu_\star}{\e^2} \|\Phi^\perp(t)\|_{\H_1}^2 - \mu_\star \|\Phi(t)\|_{\H_1}^2 
	+ \frac{C}{\e} \|h^1(t)\|_\H \|h^1(t)\|_{\H_1}\|\Phi^\perp(t)\|_\H + \frac{C}{\e^2} \|h^0(t)\|_\E \|\Phi(t)\|_\H.
$$
Making an appropriate use of Young inequality to treat the third term, we obtain that for~$\mu\in(0,\mu_\star)$, 
	\begin{equation*}\begin{split}
	\frac12 \frac{\d}{\d t} \|\Phi(t)\|^2_\H 
	&\leq - \frac{\mu}{\e^2} \|\Phi^\perp(t)\|^2_{\H_1} - \mu_\star \|\Phi(t)\|_{\H_1}^2 +C \|h^1(t)\|^2_\H \|h^1(t)\|^2_{\H_1} + \frac{C}{\e^2} \|h^0(t)\|_\E \|\Phi(t)\|_\H \\
	&\leq - \mu_\star \|\Phi(t)\|_{\H_1}^2 +C \|h^1(t)\|^2_\H \|h^1(t)\|^2_{\H_1} + \frac{C}{\e^2} \|h^0(t)\|_\E \|\Phi(t)\|_\H. 
	\end{split}\end{equation*}
In the second term, we decompose $h^1$ into two parts: $h^1 = \mathbf{P}_0 h^1 + \Phi$ and use that $\mathbf{P}_0 = \mathbf{P}_0^2$ together with the fact that $\mathbf{P}_0 \in \mathscr{B}(\E,\H)$ to obtain 
	$$
	\|h^1(t)\|^2_\H \|h^1(t)\|^2_{\H_1} \lesssim \Delta_0^2 \left(\|h^1(t)\|_\H^2 + \|\Phi(t)\|^2_{\H_1}\right). 
	$$
We can thus conclude the proof by taking $\Delta_0$ small enough and integrating the above differential inequality. Notice that the inequality stated in the Proposition holds for the equivalent ``hypocoercive norm'' introduced above and thus also holds for the usual norm on $\H$ because of the equivalence (uniformly in $\e$) between those two norms.
\qed

Combining estimates from Propositions~\ref{prop:h0} and~\ref{prop:Phi}, one can obtain that 
\begin{cor}\label{cor:Phi} 
Assume that {$\ho\in \E$, {$\hu \in \H$}} are such that
	\begin{equation*}
	\sup_{t\geq0}\big(\|\ho (t)\|_{\E}  + \|\hu(t)\|_{{{\H}}} \big) \leq \Delta_0 <\infty.
	\end{equation*}
For $\mu \in (0,\mu_\star)$ (where $\mu_\star$ is defined in~\eqref{eq:hypocoerc}), for $\Delta_0$ small enough and for any $\delta>0$, we have that:
\begin{equation}
	\begin{split}\label{eq:estimPhi2}
	\|(\mathbf{Id} - \mathbf{P}_0)h^1(t)\|^2_{\H} \lesssim 
	\frac1\delta \, \|h^{\mathrm{in}}\|^2_\E \, e^{-\mu t} + 
	(\Delta_0^2+\delta+\lambda_\e) \int_0^t e^{-\mu(t-s)} \|h^1(s)\|_\H^2 \, \d s. 
	\end{split}
\end{equation}
\end{cor}

\begin{nb}
The fact that we are able to obtain a multiplicative constant that can be chosen small in front of the second term is very important to recover a decay for $h^1$. Indeed, in Proposition~\ref{prop:P0h1}, in the estimate of $\mathbf{P}_0h^1$, the term 
	$$
	\lambda_\e \int_0^t e^{-\overline \lambda_\e (t-s)} \|(\mathbf{Id} - \mathbf{P}_0)h^1(s)\|_\H \, \d s 
	$$
is problematic when applying Gronwall lemma if one hopes to recover some decay in time but the extra small constant that appears in the estimate of $(\mathbf{Id} - \mathbf{P}_0)h^1$ in~\eqref{eq:estimPhi2} allows us to circumvent this difficulty. 
\end{nb}
In the end, we are able to prove the following proposition:

\begin{cor} \label{cor:h1}
Let $r \in (0,1)$. Assume that {$\ho\in \E$, {$\hu \in \H$}} are such that
	\begin{equation*}
	\sup_{t\geq0}\big(\|\ho (t)\|_{\E}  + \|\hu(t)\|_{{{\H}}} \big) \leq \Delta_0 <\infty
	\end{equation*}
where $\Delta_0$ is small enough so that the conclusion of Corollary~\ref{cor:Phi} holds. There exists $\e_6 \in (0,\e_5)$ (where $\e_5$ is defined in Proposition~\ref{prop:h0}) and $\lambda_6 \in (0, \lambda_4)$ (where $\lambda_4$ is defined in Lemma~\ref{lem:PaP0}) such that for any $\e \in (0,\e_6)$ and any $\lambda_0 \in [0,\lambda_6)$ (where $\lambda_0$ is defined in Assumption~\ref{hyp:re}), 
	$$
	\|h^1(t)\|_\H \leq C \, \|h^{\mathrm{in}}\|_\H \, e^{-(1-r)\overline \lambda_\e t}
	$$
where $\overline \lambda_\e$ has been introduced in Proposition~\ref{prop:P0h1} and the constant $C$ depends on $r$, $\Delta_0$, $\mu_\star$ (defined in~\eqref{eq:hypocoerc}) and $\nu_0$ (defined in Lemma~\ref{lem:reghypo}). 
\end{cor}

\subsection{Estimates on the kinetic problem}
Combining the previous corollary with Proposition~\ref{prop:h0}, we are able to get our final a priori estimates on~$h$ in the space $\E$:
\begin{prop} \label{prop:h}
Let $r \in (0,1)$. Assume that {$\ho\in \E$, {$\hu \in \H$}} are such that
	\begin{equation*}
	\sup_{t\geq0}\big(\|\ho (t)\|_{\E}  + \|\hu(t)\|_{{{\H}}} \big) \leq \Delta_0 <\infty
	\end{equation*}
where $\Delta_0$ is small enough so that the conclusion of Corollary~\ref{cor:Phi} holds. There exists $\e^\dagger \in (0,\e_6)$, $\lambda^\dagger \in (0, \lambda_6)$ (where $\e_6$ and $\lambda_6$ are defined in Proposition~\ref{cor:h1}) such that for any $\e \in (0,\e^\dagger)$ and any $\lambda_0 \in [0,\lambda^\dagger)$ (where $\lambda_0$ is defined in Assumption~\ref{hyp:re}), 
	\begin{equation*}
	\|h(t)\|_\E \leq C \, \|h^{\mathrm{in}}\|_\E \, e^{-(1-r) \lambda_\e t} \qquad
	 \text{and} \quad 
	\int_0^t \|h(s)\|_{\E_1} \, \d s \leq C \, \|h^{\mathrm{in}}\|_\E \min\left\{ 1 + t, 1 + \frac{1}{\lambda_\e} \right\}
	\end{equation*}
where $\lambda_\e \underset{\e\to 0}{\sim} (1-\re(\e))/\e^2$ has been defined in Theorem~\ref{theo:linear} and the constant $C$ depends on $r$, $\Delta_0$, $\mu_\star$ (defined in~\eqref{eq:sgG1eps}) and $\nu_0$ (defined in Lemma~\ref{lem:reghypo}). 
\end{prop}

\begin{nb}
Notice that for a fixed $\e>0$, the second a priori estimate shows that $h=h_\e$ belongs to the space~$L^1([0,\infty),\E_1)$. If one is interested in getting bounds on the family $\{h_\e\}_\e$, then we obtain that if $\lambda_0>0$ (in Assumption~\ref{hyp:re}), then the family is bounded in $L^1([0,\infty),\E_1)$ and if $\lambda_0=0$, then for any $T>0$, the family is bounded in $L^1([0,T),\E_1)$.
\end{nb}

Thanks to the above a priori estimates, we can prove Theorem~\ref{theo:main-cauchy} by introducing a suitable iterative scheme that is stable and convergent. We refer to~\cite[Section~5]{ALT} for the details of the proof. We can actually prove the following more precise estimates (which will be useful in what follows) on $h^0_\e$ and $h^1_\e$ that are respectively solutions to~\eqref{eq:h0} and~\eqref{eq:h1}: 
	\begin{equation} \label{eq:estimh0}
	\|h^0_\e\|_{L^\infty([0,\infty)\,;\,\E)} \lesssim 1
	\quad \text{and} \quad 
	\|h^0_\e\|_{L^1([0,\infty)\,;\,\E_1)} \lesssim \e^2 
	\end{equation}
as well as
	\begin{equation} \label{eq:estimh1}
	\|h^1_\e\|_{L^\infty([0,\infty)\,;\,\H)} \lesssim 1
	\quad \text{and} \quad 
	\|h^1_\e\|_{L^2([0,\infty)\,;\,\H_1)} \lesssim 1
	\end{equation}
where we recall that the spaces $\H$ and $\H_1$ are respectively defined in~\eqref{def:H} and~\eqref{def:H1}. 
Notice that in the previous inequalities, the multiplicative constants only involve quantities related to the initial data of the problem and are independent of~$\e$. 
%%%%%%%%%%%%%%%%%%%%%%%%%%%%%%%%%%%%%%%%%%%%%%%%%%%%%%%%%%%
\section{Derivation of the fluid limit system} \label{sec:hydro}

The Cauchy theory developed in the previous results give all the \emph{a priori} estimates that will allow to prove Theorem~\ref{theo:hydro}. To this end, we make additional assumptions in the definition of the spaces~$\E$ and~$\E_1$, namely, in this section, those spaces are defined through: 
	\begin{equation} \label{defbis:E-E1}
	\E:= {\W^{k,1}_{v}\W^{m,2}_{x}}(\langle v\rangle^{q}), \quad \E_{1}:={\W^{k,1}_{v}\W^{m,2}_{x}}(\langle v\rangle^{q+1})
	\quad \text{with} \quad
	 m  > d, \quad m-1 \geq  k \geq 1, \quad q   \geq 5. 
	\end{equation}
We assume that Assumption~\ref{hyp:re} is met, consider $\e$, $\lambda_0$ and $\eta_0$ sufficiently small so that the conclusion of Theorem~\ref{theo:main-cauchy} holds in those spaces and consider $\{h_\e\}_\e$ a family of solutions to~\eqref{eq:BE} constructed in this theorem that splits as $h_\e = \ho_\e+\hu_\e$ with $\ho_\e$ and $\hu_\e$ defined in Section~\ref{sec:nonlinear}. 
We also fix~$T>0$ for the rest of the section.
%-------------%-------------%-------------%-------------%-------------%-------------%-------------%-------------%
\subsection{Weak convergence}
We start by the following lemma which in particular tells that the microscopic part of $h_\e$ vanishes in the limit~$\e \to 0$:
\begin{lem} \label{lem:hperpholder}
For any $0 \leq t_1 \leq t_2 \leq T$, there holds:
	\begin{equation} \label{eq:hperpholder}
	\int_{t_1}^{t_2} \|(\mathbf{Id} - \bm{\pi}_0)h_\e(\tau)\|_{\E} \, \d \tau \lesssim \e \sqrt{t_2-t_1},
	\end{equation}
%and
%	\begin{equation} \label{eq:hperpvanish}
%	\int_{t_1}^{t_2} \|(\mathbf{Id} - \bm{\pi}_0)h_\e(\tau)\|_{\E_1} \, \d \tau \lesssim \e (1+\sqrt{t_2-t_1}),
%	\end{equation}
where we recall that $\bm{\pi}_0$ is the projection onto the kernel of $\mathscr{L}_1$ defined in~\eqref{def:pi0}. 
\end{lem}
\begin{proof}
%We only prove the first estimate, the proof of the second one is completely similar. 
We first remark that 
	\begin{multline*}
	\int_{t_1}^{t_2} \|(\mathbf{Id} - \bm{\pi}_0)h_\e(\tau)\|_{\E} \, \d\tau
	\lesssim 
	\left(\int_{t_1}^{t_2} \|(\mathbf{Id} - \bm{\pi}_0)h^0_\e(\tau)\|^2_{\E} \, \d \tau\right)^{1/2} \sqrt{t_2-t_1} \\
	+ \left(\int_{t_1}^{t_2} \|(\mathbf{Id} - \bm{\pi}_0)h^1_\e(\tau)\|^2_{\H_1} \, \d\tau\right)^{1/2} \sqrt{t_2-t_1}.
	\end{multline*}
The first term is estimated thanks to~\eqref{eq:estimh0}, which gives:
	$$
	\int_{t_1}^{t_2} \|(\mathbf{Id} - \bm{\pi}_0)h^0_\e(\tau)\|^2_{\E} \, \d \tau
	\lesssim \|(\mathbf{Id} - \bm{\pi}_0)h^0_\e\|_{L^\infty((0,T) \, ; \, \E)} 
	\|(\mathbf{Id} - \bm{\pi}_0)h^0_\e\|_{L^1((0,T) \, ; \, \E_1)} \lesssim \e^2.  
	$$
Concerning the second one, we perform similar computations as in the proof of Proposition~\ref{prop:Phi}. We recall that $h^1_\e$ solves~\eqref{eq:h1}
and consider $\|\cdot\|_\H$ an hypocoercive norm on $\H$ (see~\eqref{eq:hypocoerc}). We then have for $\mu \in (0,\mu_\star)$:
	\begin{equation*}
	\frac12 \frac{\d}{\d t} \|h^1_\e(t)\|_\H^2 
	\leq - \frac{\mu}{\e^2} \|(\mathbf{Id}- \bm{\pi}_0) h^1_\e(t)\|_{\H_1}^2
	 - \mu_\star \| h^1_\e(t)\|_{\H_1}^2 
	 +C \|h^1_\e(t)\|^2_\H \|h^1_\e(t)\|^2_{\H_1} + \frac{C}{\e^2} \|h^0_\e(t)\|_\E \|h^1_\e(t)\|_\H
	\end{equation*}
from which we deduce that 
	\begin{multline*}
	\frac{1}{\e^2} \int_{t_1}^{t_2} \|(\mathbf{Id} - \bm{\pi}_0)h^1_\e(\tau)\|^2_{\H_1} \, \d\tau 
	\lesssim \|h^1_\e(t_1)\|^2_{\H} \\
	+ \int_{t_1}^{t_2}  \|h^1_\e(\tau)\|^2_\H \|h^1_\e(\tau)\|^2_{\H_1} \, \d\tau  
	+ \frac{1}{\e^2}\int_{t_1}^{t_2} \|h^0_\e(\tau)\|_\E \|h^1_\e(\tau)\|_\H \, \d \tau \lesssim 1
	\end{multline*}
where we used~\eqref{eq:estimh0} and~\eqref{eq:estimh1} to get the last estimate. Therefore, as for $h^0_\e$ one has
$$\int_{t_1}^{t_2}\|(\mathbf{Id} - \bm{\pi}_0)h^1_\e(\tau)\|^2_{\H_1} \, \d\tau \lesssim \e^2 $$
and this allows to conclude to the wanted estimate. 
\end{proof}
Using estimates~\eqref{eq:estimh0},~\eqref{eq:estimh1} and~\eqref{eq:hperpholder}, one can prove the following result of weak convergence (we refer to~\cite[Theorem~6.4]{ALT} for more details on the proof):
\begin{theo}\label{theo:conv} 
%Fix $T >0$, and let 
%	$$
%	\left\{h_{\e}\right\}_{\e}\subset L^\infty\big((0,T)\,;\,\E) \cap L^1\big((0,T)\,;\,\E_1\big)
%	$$ 
%be a sequence of solutions to the inelastic Boltzmann equation~\eqref{eq:BE} constructed in Theorem~\ref{theo:main-cauchy}.  Then, with the splitting~$h_{\e}=\ho_{\e}+\hu_{\e}$, 
Up to extraction of a subsequence, one has
	\begin{equation}\begin{cases}\label{eq:mode-conv}
	\left\{\ho_{\e}\right\}_{\e} \text{converges to $0$ strongly  in } L^{1}((0,T)\,;\,{\E_1}), \\[0.2cm]
	\left\{\hu_{\e}\right\}_{\e} \text{converges to $\bm{h}$ weakly in } L^{2}\left((0,T)\,;\H\right),
	\end{cases}\end{equation}
where $\bm{h}=\bm{\pi}_{0}(\bm{h})$. In particular, there exist 
	\begin{equation*}
	\varrho \in L^{2}\left((0,T);\,\W^{m,2}_{x}(\T^{d})\right), \quad
	u \in L^{2}\left((0,T);\;\left(\W^{m,2}_{x}(\T^{d})\right)^{d}\right),  \qquad
	\vE \in L^{2}\left((0,T);\,\W^{m,2}_x(\T^{d})\right), 
	\end{equation*}
such that
	\begin{equation}\label{def:bmh}
	\bm{h}(t,x,v)=\left(\varrho(t,x)+u(t,x)\cdot v + \frac{1}{2}\vE(t,x)(|v|^{2}-d\en_{1})\right)\M(v)
	\end{equation}
where $\M$ is the Maxwellian distribution introduced in \eqref{eq:max}.
\end{theo}
\begin{nb}
Recall that $(\varrho,u,\theta)$ can be expressed in terms of $\bm{h}$ through the following equalities:
	\begin{equation} \label{def:rho-u-theta} 
	\varrho(t,x) = \int_{\R^d} \bm{h}(t,x,v) \, \d v, 
	\quad
	u(t,x) = \frac{1}{\vartheta_1} \int_{\R^d} \bm{h}(t,x,v) v \, \d v, 
	\quad
	\vE(t,x) =  \int_{\R^d} \bm{h}(t,x,v) \frac{|v|^2-d\vartheta_1}{\vartheta_1^2d}\, \d v. 
	\end{equation}
\end{nb}

%-------------%-------------%-------------%-------------%-------------%-------------%-------------%-------------%
\subsection{Limit system}
As mentioned in the introduction, the path that we use to derive the limit system follows the same lines as in the elastic case. The main idea is to write equations satisfied by averages in velocity of~$h_\e$ and to study the convergence of each term. It is worth mentioning that with the notion of weak convergence  at hand presented above, we can adopt an approach which is reminiscent of the program established in~\cite{BaGoLe1,BaGoLe2} but simpler. In particular, we can adapt some of the main ideas of~\cite{golseSR} regarding the delicate convergence of nonlinear terms. The detailed computations and arguments are included in~\cite[Section~6]{ALT}, we only mention the main steps and keypoints of the proof hereafter. 
In what follows, we will use the following notation: for~$g=g(x,v)$, $$\langle g \rangle := \int_{\R^d} g(\cdot,v) \, \d v.$$ 

\subsubsection*{\textit{\textbf{Local conservation laws}}}
We introduce 
	\begin{equation} \label{def:bmA}
	\bm{A} (v) := v \otimes v - \frac1d |v|^2 \mathbf{Id} \quad \text{and} \quad p_\e := \la \frac1d |v|^2 h_\e \ra
	\end{equation}
so that  $
	\la v \otimes v \, h_\e \ra = \la \bm{A} \, h_\e \ra + p_\e \, \mathbf{Id}.
	$
We integrate in velocity equation~\eqref{eq:BE} multiplied by $1$, $v_{i}$, $\frac{1}{2}\,|v|^{2}$, to obtain
	\begin{subequations}\label{moments}
	\begin{equation}\label{eq:mass}
	\partial_{t}\la h_{\e}\ra +\frac{1}{\e}\mathrm{div}_{x}\la v\,h_{\e}\ra=0,
	\end{equation}

	\begin{equation}\label{eq:bulk}
	\partial_{t} \la v\,h_{\e}\ra +\frac{1}{\e}
	\mathrm{Div}_{x}\la \bm{A} \, h_{\e} \ra  + \frac1\e \nabla_x p_\e =\frac{1-\re(\e)}{\e^{2}}\la v\,h_{\e}\ra,
	\end{equation}

	\begin{equation}\label{eq:energy}
	\partial_{t}\la \tfrac{1}{2}|v|^{2}h_{\e}\ra+\frac{1}{\e}\mathrm{div}_{x}\,\la \tfrac{1}{2}|v|^{2}v\,h_{\e}\ra\,=\frac{1}{\e^{3}}		\mathscr{J}_{\re(\e)}(f_{\e},f_{\e})+\frac{2(1-\re(\e))}{\e^{2}}\la \tfrac{1}{2}|v|^{2}h_{\e}\ra,
	\end{equation}
where we recall that $f_\e=G_{\re(\e)} + \e h_\e$ and where we have introduced
	$$
	\mathscr{J}_{\re}(f,f):=\int_{\R^{d}}\left[\Q_{\re}(f,f)-\Q_{\re}(G_{\re},G_{\re})\right]\,|v|^{2}\,\d v.
	$$
	\end{subequations} 
%We thus observe that the quantities
%	$$
%	\la h_{\e}\ra,\quad \la vh_{\e}\ra,
%	\quad \la \tfrac{1}{2}|v|^{2}h_\e\ra, 
%	\quad \la \tfrac{1}{2}|v|^{2}v\,h_{\e}\ra
%	\quad \text{and} \quad \la v\otimes v\,h_{\e}\ra
%	$$
%are important.
%are important and as in the classical elastic case, we write
% 	$$
%	\la v\otimes v\,h_{\e}\ra=\la \bm{A}\,h_{\e}\ra + p_{\e}\,\mathbf{Id}, \quad p_{\e}:=\la \frac{1}{d}|v|^{2}\,h_{\e}\ra,
%	$$
%where we introduce the traceless tensor 
%	$$
%	\bm{A}(v):=v \otimes v -\frac{1}{d}|v|^{2}\mathbf{Id}.
%	$$
The goal is to study the convergence of each term in~\eqref{eq:mass}-\eqref{eq:bulk}-\eqref{eq:energy}. 
A first important remark to address this point is that thanks to the estimates recalled in~\eqref{eq:estimh0}-\eqref{eq:estimh1}, one can prove that for any function~$\psi = \psi(v)$ satisfying the bound~$|\psi(v)| \lesssim \langle v \rangle^q$, we have the following convergence in the distributional sense:
	\begin{equation} \label{eq:moments}
	\langle  \psi \, h_\e \rangle \xrightarrow[\e \to 0]{} \langle \psi \, \bm{h} \rangle
	\quad \text{in} \quad \mathscr{D}'_{t,x}
	\end{equation}
where $\bm{h}$ is defined in~\eqref{def:bmh} (see~\cite[Lemma~6.6]{ALT}). 
%We also have that 
%	$$	
%	\langle  \psi \, \left[\Q_1(h_\e,h_\e)-\Q_1(\bm{\pi}_0 h_\e,\bm{\pi}_0 h_\e)\right] \rangle 
%	\xrightarrow[\e \to 0]{} 0
%	\quad \text{in} \quad \mathscr{D}'_{t,x}.
%	$$

Roughly speaking, the convergence of the terms in the LHS of~\eqref{eq:mass}-\eqref{eq:bulk}-\eqref{eq:energy} is treated as in the elastic case. 
The RHS is going to be handled as a source term which takes into account the drift term and the dissipation of kinetic energy at the microscopic level. In this regard, using~\eqref{eq:moments}, we first remark that under Assumption~\ref{hyp:re},
	\begin{equation} \label{eq:convdrift}
	\frac{1-\re(\e)}{\e^{2}}\la v\,h_{\e}\ra \xrightarrow[\e \to 0]{} {\en_{1}}\lambda_0u
	\quad \text{in} \quad {\mathscr{D}'_{t,x}},
	\end{equation}
since $\lambda_0=\lim_{\e \to 0^{+}}\e^{-2}(1-\re(\e))$ and from the definition of $u$ in~\eqref{def:rho-u-theta}.
We then present a result of convergence for~$\e^{-3} \mathscr{J}_{\re(\e)}(f_\e,f_\e)$ in the following lemma in the proof of which there are not major difficulties. The said proof is thus omitted, we just mention that it is based on Assumption~\ref{hyp:re}, on the estimates on $h_\e$ coming from~\eqref{eq:estimh0}-\eqref{eq:estimh1} and involves the dissipation of energy~\eqref{eq:Dre}, we refer to~\cite[Lemma~6.9]{ALT} for more details.
\begin{lem}\label{lem:J0}
It holds that  
	$$
	\frac{1}{\e^{3}}\mathscr{J}_{\re(\e)}(f_{\e},f_{\e}) \xrightarrow[\e \to 0]{} \mathcal{J}_{0}
	\quad \text{in} \quad  \mathscr{D}'_{t,x},
	$$
where
	$$
	\mathcal{J}_{0}(t,x):=-\lambda_{0}\,\bar{c}\,\en_{1}^{\frac{3}{2}}\left(\varrho(t,x)+\frac{3}{4}\en_{1}\,\vE(t,x)\right)
	$$
for some positive constant $\bar{c}$ depending only on the angular kernel $b(\cdot)$ and $d$ and where $\lambda_0$ is defined in Assumption~\ref{hyp:re}. 
%In particular, 
%	$$
%	\mathcal{J}_{0} =-\lambda_{0}\,\bar{c}\,\en_{1}^{\frac{5}{2}}\left(E(t)-\frac{1}{4}\vE(t,x)\right).
%	$$
\end{lem}

\subsubsection*{\textbf{Incompressibility condition and Boussinesq relation}}
Using~\eqref{eq:moments} in the equations~\eqref{eq:mass}-\eqref{eq:bulk}, using also that the restitution coefficient satisfies Assumption~\ref{hyp:re}, we can easily obtain the incompressibility condition as well as the Boussinesq relation:
	\begin{equation} \label{eq:incomp+Boussinesq}
	\operatorname{div}_x u = 0 \quad \text{and} \quad \nabla_x (\varrho + \vartheta_1 \theta) = 0
	\end{equation}
where we recall that $\varrho$, $u$ and $\theta$ are defined in~\eqref{def:rho-u-theta}. Using furthermore that the global mass of~$h_\e$ vanishes (see~\eqref{eq:conservh}), we have that
	$$
	0=\int_{\T^d \times \R^d}h_\e(t,x,v) \, \d v \, \d x \xrightarrow[\e \to 0]{} \int_{\T^d} \varrho(t,x) \, \d x 
	\quad \text{in} \quad \mathscr{D}'_t
	$$
and thus that $\int_{\T^d} \varrho(t,x) \, \d x =0$. It implies that we have the following strengthened Boussinesq relation: for almost every $(t,x) \in (0,T) \times \T^d$, 
	\begin{equation} \label{eq:strongBoussinesq}
	\varrho + \vartheta_1(\theta - E) = 0 
	\quad \text{with} \quad E=E(t) := \int_{\T^d} \theta(t,x) \, \d x.
	\end{equation}
\begin{nb}
Notice here that the derivation of the strong Boussinesq relation $\varrho + \vartheta_1 \theta=0$ is not as straightforward as in the elastic case. In the elastic case, the classical Boussinesq relation $\nabla_x (\varrho + \vartheta_1 \theta) = 0$ straightforwardly implies the strong form of Boussinesq because the two functions~$\varrho$ and~$\vE$ have zero spatial averages. This cannot be deduced directly in the granular context due to the dissipation of energy and we will see later on how to obtain it (see Proposition~\ref{prop:limit2}). 
\end{nb}

\subsubsection*{\textbf{\textit{Equations of motion and temperature}}}
In order to identify the equations satisfied by~$u$ and~$\theta$, as in the elastic case, we start by studying the convergence of quantities that are related to 
\begin{multline} \label{def:rhoeps-ueps-thetaeps} 
	\varrho_\e(t,x) := \int_{\R^d} h_\e(t,x,v) \, \d v, 
	\quad
	u_\e(t,x) := \frac{1}{\vartheta_1} \int_{\R^d} h_\e(t,x,v) v \, \d v, 
	\quad
	\vE_\e(t,x) :=  \int_{\R^d} h_\e(t,x,v) \frac{|v|^2-d\vartheta_1}{\vartheta_1^2d}\, \d v. 
	\end{multline}
More precisely, we inverstigate the convergence of 
	$$
	\bm{u}_\e := \exp\left(-t \frac{1-\re(\e)}{\e^2}\right) \mathcal{P} u_\e 
	\quad \text{and} \quad
	\bm{\theta}_\e :=  \la \tfrac12 (|v|^2-(d+2)\vartheta_1) h_\e \ra
	$$
where $\mathcal{P}$ is the Leray projection on divergence-free vector fields. 
Notice that if we compare our approach to the elastic case, we have added the exponential term in the definition of~$\bm{u}_\e$ in order to absorbe the term in the RHS in~\eqref{eq:bulk}. We compute the evolution of~$\bm{u}_\e$ and~$\bm{\theta}_\e$ (by applying the Leray projector~$\mathcal{P}$ to~\eqref{eq:bulk} and by making an appropriate linear combination of~\eqref{eq:mass} and~\eqref{eq:energy}) and obtain:
	\begin{equation} \label{eq:bulk2}
	\partial_{t} \bm{u}_{\e}
	=-\exp\left(-t\frac{1-\re(\e)}{\e^{2}}\right) 
	\mathcal{P}\left(\en_{1}^{-1}\mathrm{Div}_{x}\la \tfrac{1}{\e}\bm{A}\,h_{\e} \ra\right)
	\end{equation}
where $\bm{A}$ is defined in~\eqref{def:bmA} and
	\begin{multline} \label{eq:energy2}
	\partial_{t}\bm{\theta}_{\e}+\frac{1}{\e}\mathrm{div}_{x}\la \bm{b}\,h_{\e}\ra
	=\frac{1}{\e^{3}}\mathscr{J}_{\re(\e)}(f_{\e},f_{\e})+\frac{2(1-\re(\e))}{\e^{2}}\la \tfrac{1}{2}|v|^{2}h_{\e}\ra \\
	\quad \text{with} \quad 
	\bm{b}(v):=\frac{1}{2}\left(|v|^{2}-(d+2)\en_{1}\right).
	\end{multline}
The study of the limit~$\e \to 0$ in those equations is more favorable because compared to~\eqref{eq:mass}-\eqref{eq:bulk}-\eqref{eq:energy}, the gradient term in~\eqref{eq:bulk} has been eliminated thanks to the Leray projector and also because $\bm{A}$ and~$\bm{b}$ belong to the range of~$\mathbf{Id} - \bm{\pi}_0$ so that thanks to Lemma~\ref{lem:hperpholder}, we know that the quantities $\e^{-1} \mathrm{Div}_{x}\la \bm{A}\,h_{\e} \ra$ and~$\e^{-1}\mathrm{div}_{x}\la \bm{b}\,h_{\e}\ra$ are bounded in~$\W^{m-1,2}_x$. Then, applying a precised version of Aubin-Lions lemma ~\cite[Corollary 4]{simon}, we are able to prove that up to the extraction of a subsequence,~$\{\bm{u}_\e\}_\e$ and~$\{\bm{\theta}_\e\}_\e$ converge strongly in~$L^1 \left((0,T) \,;\, \W^{m-1,2}_x\right)$ respectively towards 
	\begin{equation} \label{def:bmtheta0}
	\mathcal{P} u = u 
	\quad \text{and} \quad
	\bm{\theta}_0 := \la \tfrac12 (|v|^2-(d+2)\vartheta_1) \bm{h} \ra = \frac{d \vartheta_1^2}{2} E - \frac{d+2}{2} \vartheta_1 \varrho
	\end{equation}
where we used the incompressibility condition and the strong Boussinesq relation given in~\eqref{eq:incomp+Boussinesq}-\eqref{eq:strongBoussinesq}. We refer to~\cite[Lemma~6.10]{ALT} for more details. 

\subsubsection*{\textit{\textbf{About initial data}}} Recall that, in Theorem \ref{theo:conv}, the convergence of~${h}_{\e}$ to~$\bm{h}$ given by~\eqref{def:bmh} is known to hold only for a subsequence and, in particular, at initial time, different subsequences could converge towards different initial datum and therefore $(\varrho,u,\vE)$ could be different solutions to the same system. In Theorem~\ref{theo:hydro}, the initial datum is prescribed by ensuring the convergence of $\bm{\pi}_0h_{\mathrm{in}}^{\e}$ towards a \emph{single} possible limit where $\bm{\pi}_0$ is defined in~\eqref{def:pi0} (recall that the initial data for~$(\varrho,u,\theta)$ is defined in~\eqref{eq:DI}). 

Using Lemma~\ref{lem:hperpholder}, one can apply Arzel\`a-Ascoli theorem to get that $\mathcal{P} u_\e$ and $\bm{\theta}_\e$ converge strongly in $\mathcal{C}\big([0,T] \, ; \, \W^{m-1,2}_x\big)$ towards respectively $u$ and $\bm{\theta}_0$ defined in~\eqref{def:bmtheta0} that also belong to $\mathcal{C}\big([0,T] \, ; \, \W^{m-1,2}_x\big)$. We refer to~\cite[Proposition~6.19]{ALT} for more details.

\subsubsection*{\textit{\textbf{Limit equations}}}
To get the limit equations, we need to study the convergence of the terms $\e^{-1} \mathcal{P} \mathrm{Div}_{x}\la \bm{A}\,h_{\e} \ra$ and $\e^{-1}\mathrm{div}_{x}\la \bm{b}\,h_{\e}\ra$ in~\eqref{eq:bulk2} and~\eqref{eq:energy2}. To this end, our approach relies on arguments coming from~\cite{golseSR} (in particular, the tricky convergence of the nonlinear terms is treated thanks to a compensated compactness argument coming from~\cite{lions-masm}), the main difference being that we force the elastic collision operator to appear in our computations, we thus introduce terms that involve differences between the elastic and the inelastic collision operators. Those remainder terms vanish in the limit $\e \to 0$ thanks to Assumption~\ref{hyp:re}. We refer to~\cite[Lemmas~6.12-6.13-6.14]{ALT} for more details. 
In the end, writing $\mathcal{P}\mathrm{Div}_{x}(u \otimes u)=\mathrm{Div}_{x}(u \otimes u) + \en_1^{-1}\nabla_{x} p$ (see \cite[Proposition 1.6]{majda}), we obtain the following result:
\begin{prop} \label{prop:limit1}
There are some constants $\nu>0$ and $\gamma>0$ such that the limit velocity~$u=u(t,x)$ in~\eqref{def:bmh} satisfies
	\begin{equation} \label{eq:bulk3}
	\partial_{t}u-\frac{\nu}{\en_1}\,\Delta_{x}u 
	+ \en_{1} \mathrm{Div}_{x} \left(u\otimes u\right)+\nabla_{x}p=\lambda_0u
	\end{equation}
where $\lambda_0$ is defined in Assumption~\ref{hyp:re}, while the limit temperature $\vE=\vE(t,x)$ in~\eqref{def:bmh} satisfies 
	\begin{equation} \label{eq:temperature}
	\partial_{t}\vE 
	-\frac{\gamma}{\en_{1}^{2}}\,\Delta_{x}\vE + \en_{1}\,u\cdot \nabla_{x}\vE
	=\frac{2}{(d+2)\en_{1}^{2}}\mathcal{J}_{0}+\frac{2d\lambda_0}{d+2}E+\frac{2}{d+2}\frac{\d}{\d t}E,
	\end{equation}
where we recall that $\mathcal{J}_0$ is defined in Lemma~\ref{lem:J0} and $E$ is defined in~\eqref{eq:strongBoussinesq}. 
\end{prop} 
\begin{nb}
The viscosity and heat conductivity coefficients~$\nu$ and~$\gamma$ are explicit and fully determined by the elastic linearized collision operator~$\mathscr{L}_1$ (see~\cite[Lemma~C.1]{ALT}).  
Notice also that, due to~\eqref{eq:incomp+Boussinesq}, $\mathrm{Div}_{x}(u\otimes u)=\left(u \cdot \nabla_{x}\right)u$ and \eqref{eq:bulk3} is nothing but a {\emph{reinforced}} Navier-Stokes equation associated to a divergence-free source term given by $\lambda_0u$ which  can be interpreted as  an energy supply/self-consistent force acting on the hydrodynamical system because of the self-similar rescaling.
\end{nb}

To end the identification of the limit equations, we go back to the strong Boussinesq equation~\eqref{eq:strongBoussinesq} and prove the following result:
\begin{prop}\label{prop:limit2} It holds that
	$$
	E(t)=0, \quad t \in [0,T],
	$$
where $E=E(t)$ is defined in~\eqref{eq:strongBoussinesq}. Consequently, the limiting temperature $\vE(t,x)$ in \eqref{def:bmh} satisfies
	\begin{equation}\label{eq:temperature2}
	\partial_{t}\,\vE-\frac{\gamma}{\en_{1}^{2}}\,\Delta_{x}\vE + \en_{1}\,u\cdot \nabla_{x}\vE
	=\frac{\lambda_{0}\,\bar{c}}{2(d+2)}\sqrt{\en_{1}}\,\vE. 
	\end{equation}
where $\gamma$ is defined in Proposition~\ref{prop:limit1},~$\lambda_0$ in Assumption~\ref{hyp:re} and $\bar c$ in Lemma~\ref{lem:J0}. Moreover, the strong Boussinesq relation holds true:
	\begin{equation}\label{eq:strongBoussinesq2}
	\varrho+\en_{1}\vE=0 \quad \text{on} \quad [0,T] \times \T^d. 
	\end{equation}
\end{prop}

\begin{proof}
Using Lemma~\ref{lem:J0} and averaging in position the equation~\eqref{eq:temperature}, it is easy to prove that 
	$$
	\frac{\d}{\d t} E(t) = \bar c_0 \, E(t)
	$$
for some some constant $\bar c_0 \in \R$. 
Moreover, on the one hand, from~\eqref{eq:DI}, we have
	\begin{equation} \label{eq:E(0)1}
	E(0) =  \int_{\T^d} \theta (0,x) \, \d x = - \frac{1}{\vartheta_1} \int_{\T^d} \varrho(0,x) \, \d x.
	\end{equation}
On the other hand, from the definition of $\bm{\theta}_0$ in~\eqref{def:bmtheta0}, we also have
	\begin{equation} \label{eq:E(0)2}
	E(0) = \frac{2}{\vartheta_1^2 d}  \int_{\T^d} \bm{\theta}_0(0,x) \, \d x 
	+ \frac{2}{\vartheta_1 d}  \int_{\T^d} \varrho(0,x) \, \d x.
	\end{equation}
 We also know that $\bm{\theta}_\e$ converges towards $\bm{\theta}_0$ in $\mathcal{C}\big([0,T] \, ; \, \W^{m-1,2}_x\big)$. Consequently, we deduce that 
 	\begin{equation*}
	\int_{\T^d} \bm{\theta}_0(0,x) \, \d x = \lim_{\e \to 0} \int_{\T^d} \la \tfrac{|v|^2-(d+2) \vartheta_1}{2} h_\e(0,x) \ra\, \d x 
	= \lim_{\e \to 0}  \int_{\T^d} \la \tfrac12 |v|^2 h_\e(0,x) \ra \, \d x
	\end{equation*}
where we used~\eqref{eq:conservh} to get the last equality. 
{From~\eqref{eq:initialenergy}, we deduce that}
	\begin{equation*}
	\int_{\T^d} \bm{\theta}_0(0,x) \, \d x=0.
	\end{equation*}
Coming back to~\eqref{eq:E(0)1}-\eqref{eq:E(0)2}, we deduce that 
	$$
	E(0) = - \frac{1}{\vartheta_1} \int_{\T^d} \varrho(0,x) \, \d x = \frac{2}{\vartheta_1 d}  \int_{\T^d} \varrho(0,x) \, \d x
	$$
which implies that $E(0)=0$. This concludes the proof. 
\color{black}
\end{proof}
Gathering the results we obtained in Propositions~\ref{prop:limit1} and~\ref{prop:limit2}, we are able to end the proof of Theorem~\ref{theo:hydro}. 
%%%%%%%%%%%%%%%%%%%%%%%%%%%%%%%%%%%%%%%%%%%%%%%%%%%%%%%%%%%
 
\end{document}